\documentclass[11pt]{article}

\usepackage{amsmath}
\usepackage{amssymb}
\usepackage{amsthm}

\usepackage{titlesec}
\usepackage{enumitem}
\usepackage{cases}

\newtheorem{thm}{Theorem}[section]
\newtheorem{lemma}[thm]{Lemma}
\newtheorem{prop}[thm]{Proposition}
\newtheorem{conj}[thm]{Conjecture}
\newtheorem{cor}[thm]{Corollary}
\theoremstyle{definition}
\newtheorem{defin}[thm]{Definition}

\title{\normalsize{\uppercase{\bf{Nodal intersections for random waves against a segment on the 3-dimensional torus}}}}
\author{\small{\uppercase{Riccardo W. Maffucci}\footnote{King's College London, Strand, London WC2R 2LS, England, United Kingdom.\newline riccardo.maffucci@kcl.ac.uk}}}
\date{}

\begin{document}
\titleformat{\section}
  {\normalfont\scshape\centering\bf}{\thesection}{1em}{}
\titleformat{\subsection}
  {\normalfont\scshape\bf}{\thesubsection}{1em}{}
\numberwithin{equation}{section}
\maketitle

\begin{abstract}
We consider random Gaussian eigenfunctions of the Laplacian on the three-dimensional flat torus, and investigate the number of nodal intersections against a straight line segment. The expected intersection number, against any smooth curve, is universally proportional to the length of the reference curve, times the wavenumber, independent of the geometry. We found an upper bound for the nodal intersections variance, depending on the arithmetic properties of the straight line. The considerations made establish a close relation between this problem and the theory of lattice points on spheres.
\end{abstract}
{\bf Keywords:} nodal intersections, arithmetic random waves, Gaussian eigenfunctions, lattice points on spheres.
\\
{\bf MSC(2010):} 11P21, 60G15.

\section{Introduction}
\subsection{Nodal intersections and lattice points on spheres}
On the three-dimensional flat torus $\mathbb{T}^3:=\mathbb{R}^3/\mathbb{Z}^3$ consider a real-valued eigenfunction of the Laplacian $F:\mathbb{T}^3\to\mathbb{R}$, with eigenvalue $\lambda^2$:
\begin{equation*}
(\Delta+\lambda^2) F=0.
\end{equation*}
The nodal set of $F$ is the zero locus
\begin{equation*}
\{x\in\mathbb{T}^3 : F(x)=0\},
\end{equation*}
consisting of a union of smooth surfaces, possibly together with a set of lower dimension, i.e. curves and points (cf. \cite{cheng1}, \cite{rudwi2}).
\\
Let $\mathcal{C}\subset \mathbb{T}^3$ be a fixed straight line segment on the torus, of length $L$, parametrised by $\gamma(t)=t\alpha=t(\alpha_1,\alpha_2,\alpha_3)$, with $0\leq t\leq L$, $\alpha\in\mathbb{R}^3$ and $|\alpha|=1$.
\\
We want to study the number of nodal intersections
\begin{equation}
\label{Zdet}
\#\{x\in\mathbb{T}^3 : F(x)=0\} \cap \mathcal{C},
\end{equation}
i.e., the number of zeros of $F$ on $\mathcal{C}$, as $\lambda\to\infty$.

The Laplace eigenvalues (``energy levels") on $\mathbb{T}^3$ are $\lambda^2=4\pi^2 m$, where $m$ is a natural number expressible as a sum of three integer squares. Let
\begin{equation}
\label{E}
\mathcal{E}=\mathcal{E}(m):=\{\mu=(\mu_1,\mu_2,\mu_3)\in\mathbb{Z}^3 : \mu_1^2+\mu_2^2+\mu_3^2=m\}
\end{equation}
be the set of all lattice points on the sphere of radius $\sqrt{m}$. Their cardinality equals $r_3(m)$, the number of ways that $m$ can be written as a sum of three squares, and will be denoted
\begin{equation*}
N=N_m:=\#\mathcal{E}=r_3(m)
\end{equation*}
(see Section \ref{seclp}); it is also the dimension of the eigenspace relative to the eigenvalue $4\pi^2 m$.
The eigenspace admits the $L^2$-orthonormal basis $\{e^{2\pi i\langle\mu,x\rangle}\}_{\mu\in\mathcal{E}}$, a general form of (complex-valued) eigenfunctions being
\begin{equation*}
F(x)=
\sum_{\mu\in\mathcal{E}}
c_{\mu}
e^{2\pi i\langle\mu,x\rangle},
\end{equation*}
with $c_{\mu}\in\mathbb{C}$ Fourier coefficients. We will henceforth consider only real-valued eigenfunctions.

\subsection{Arithmetic random waves}
One cannot expect to have any deterministic lower or upper bounds for the number of nodal intersections \eqref{Zdet}. Indeed, \cite[Examples 1.1, 1.2]{ruwiye} gives sequences of eigenfunctions $F$ and curves $\mathcal{C}$ where $\mathcal{C}$ is contained in the nodal set for arbitrarily high energy,
and planar curves with no nodal intersections at all, $m$ arbitrarily large. Let us then consider the {\em random} Gaussian toral eigenfunctions (`arithmetic random waves' \cite{orruwi}, \cite{rudwi2}, \cite{krkuwi})
\begin{equation}
\label{arw}
F(x)=\frac{1}{\sqrt{N}}
\sum_{(\mu_1,\mu_2,\mu_3)\in\mathcal{E}}
a_{\mu}
e^{2\pi i\langle\mu,x\rangle}
\end{equation}
with eigenvalue $\lambda^2=4\pi^2m$, where $a_{\mu}$ are complex standard Gaussian random variables\footnote{defined on some probability space $(\Omega,\mathcal{F},\mathbb{P})$, where $\mathbb{E}$ denotes
the expectation with respect to $\mathbb{P}$.} ($\mathbb{E}(a_{\mu})=0$ and $\mathbb{E}(|a_{\mu}|^2)=1$), independent save for the relations $a_{-\mu}=\overline{a_{\mu}}$ (so that $F(x)$ is real valued).

{\em Notation.} For functions $f$ and $g$, we will use $f=O(g)$ or $f \ll g$ interchangeably to denote the inequality $f \leq C g$ for some constant $C$. We write $f\ll_a g$ to emphasize the dependence of $C$ on the parameter $a$. The statement $f\asymp g$ means $g\ll f\ll g$.

Given a toral curve $\mathcal{C}$, we are interested in the distribution of the number of nodal intersections
\begin{equation}
\label{Z}
\mathcal{Z}=\mathcal{Z}(F):=\#\{x : F(x)=0\} \cap \mathcal{C}
\end{equation}
for an arithmetic random wave $F$. Rudnick, Wigman and Yesha \cite{ruwiye} computed the expectation to be, for any smooth curve of length $L$ on $\mathbb{T}^3$,
\begin{equation}
\label{expectation}
\mathbb{E}[\mathcal{Z}]=L\frac{2}{\sqrt{3}}\cdot\sqrt{m}.
\end{equation}
Moreover, they bounded the variance of $\mathcal{Z}$ for curves with nowhere zero curvature, assuming that $\mathcal{C}$ either has nowhere vanishing torsion or is planar:
\begin{equation*}
\text{Var}\left(\frac{\mathcal{Z}}{\sqrt{m}}\right)
\ll
\frac{1}{m^\delta}
\end{equation*}
for $m\not\equiv 0,4,7 \pmod 8$, where we may take $\delta=\frac{1}{3}$ in the case of nowhere vanishing torsion, and any $\delta<\frac{1}{4}$ for planar curves. It follows that
for all $\epsilon>0$, as $m\to\infty$, $m\not\equiv 0,4,7 \pmod 8$, the number of nodal intersections satisfies
\begin{equation}
\label{distrasymean}
\lim_{\substack{m\to\infty \\ m\not\equiv 0,4,7 \pmod 8}}\mathbb{P}\left(\left|\frac{\mathcal{Z}(F)}{\sqrt{m}}-\frac{2}{\sqrt{3}}L\right|>\epsilon\right)=0.
\end{equation}
Note that the condition $m\not\equiv 0,4,7 \pmod 8$ is natural (cf. \cite[Section 1.3]{ruwiye}): indeed, if $m\equiv 7 \pmod 8$, the set of lattice points $\mathcal{E}(m)$ \eqref{E} is empty; moreover,
\begin{equation*}
\mathcal{E}(4m)=\{2\mu : \mu\in\mathcal{E}(m)\}
\end{equation*}
(see e.g. \cite[\S 20]{harwri}), hence it suffices to consider energies $m$ up to multiples of $4$.


\subsection{Statements of the main results}
Our purpose is to investigate the nodal intersections number $\mathcal{Z}$ \eqref{Z} for {\em straight} line segments
\begin{equation}
\label{C}
\mathcal{C}: \gamma(t)=t(\alpha_1,\alpha_2,\alpha_3),
\quad
0\leq t\leq L, \quad \alpha\in\mathbb{R}^3, \quad |\alpha|=1,
\end{equation}
the other extreme of the nowhere zero curvature setting. Recall that the expected value of $\mathcal{Z}$ is given by \eqref{expectation}. In a moment we will give upper bounds for the variance, depending on whether the straight line $\mathcal{C}$ is `rational'. Given $\mathcal{C}$ as in \eqref{C}, at least one of the $\alpha_i$, say $\alpha_1$, is non-zero: we call $\alpha$ a {\em `rational vector'} if
\begin{equation*}
{\alpha_2}/{\alpha_1}\in\mathbb{Q} \quad\text{and}\quad {\alpha_3}/{\alpha_1}\in\mathbb{Q};
\end{equation*}
otherwise, we call $\alpha$ an {\em `irrational vector'}. Accordingly, we say that $\mathcal{C}$ is a {\em `rational/irrational line'}.

Let us generalise the setting, and introduce arithmetic random waves in any dimension $d\geq 1$. We denote
\begin{equation}
\label{lpsetd}
_d\mathcal{E}(m):=\{\mu\in\mathbb{Z}^d : |\mu|^2=m\}
\end{equation}
the set of all lattice points on $\sqrt{m}\mathcal{S}^{d-1}$, and $_dN_m=|{_d\mathcal{E}(m)}|$. We omit the index $d$ when $d=3$, as we deal mostly with the $3$-dimensional setting. On $\mathbb{T}^d$, we define the arithmetic random waves $F:\mathbb{T}^d\to\mathbb{R}$,
\begin{equation*}
F(x)=\frac{1}{\sqrt{_dN_m}}
\sum_{\mu\in{_d\mathcal{E}(m)}}
a_{\mu}
e^{2\pi i\langle\mu,x\rangle},
\end{equation*}
where $m$ is expressible as the sum of $d$ integer squares, and $a_{\mu}$ are complex standard Gaussian random variables, independent save for the relations $a_{-\mu}=\overline{a_{\mu}}$, which make $F(x)$ real-valued. Here and elsewhere we will denote
\begin{equation*}
R:=\sqrt{m},
\end{equation*}
and $R\mathcal{S}^{d-1}$ the $d-1$-dimensional sphere of radius $R$.
\begin{defin}[{\cite[Section 2.3]{brgafa}}]
\label{kappa}
Let $\kappa_d(R)$ be the maximal number of lattice points in the intersection of $R\mathcal{S}^{d-1}\subset\mathbb{R}^d$ and any hyperplane $\Pi$:
\begin{equation*}
\kappa_d(R)=\max_{\Pi
} \#\{\mu\in\mathbb{Z}^d : \mu\in R\mathcal{S}^{d-1}\cap\Pi\}.
\end{equation*}
\end{defin}
We denote $\kappa(R):=\kappa_3(R)$, as we will mostly be concerned with $d=3$. Jarnik (see \cite{jarnik}, or \cite[(2.6)]{brgafa}) proved the upper bound
\begin{equation}
\label{kappa3bound}
\kappa(R)\ll R^\epsilon, \quad \forall\epsilon>0.
\end{equation}
\begin{thm}
\label{resultrat}
Let the straight line segment $\mathcal{C}\subset\mathbb{T}^3$ be parametrised by
$\gamma(t)=t\alpha$, where 
$\alpha$ is a norm one rational vector. Then the nodal intersections variance has the upper bound
\begin{equation*}
\text{Var}\left(\frac{\mathcal{Z}}{\sqrt{m}}\right)\ll\frac{\kappa(\sqrt{m})}{N},
\end{equation*}
the implied constant depending only on $\alpha$.
\end{thm}
See Section \ref{rational} for the proof of Theorem \ref{resultrat}. For irrational lines we may unconditionally prove the following two theorems, distinguishing between irrational lines \eqref{C} satisfying
\begin{equation*}
{\alpha_2}/{\alpha_1}\in\mathbb{R}\setminus\mathbb{Q} \quad\text{and}\quad {\alpha_3}/{\alpha_1}\in\mathbb{R}\setminus\mathbb{Q}
\end{equation*}
and those satisfying 
\begin{equation*}
{\alpha_2}/{\alpha_1}\in\mathbb{Q} \quad\text{and}\quad {\alpha_3}/{\alpha_1}\in\mathbb{R}\setminus\mathbb{Q}.
\end{equation*}
\begin{thm}
\label{resultirrat}
Let $m\not\equiv 0,4,7 \pmod 8$ and the straight line segment $\mathcal{C}\subset\mathbb{T}^3$ be parametrised by $\gamma(t)=t(\alpha_1,\alpha_2,\alpha_3)$ with  $\frac{\alpha_2}{\alpha_1},\frac{\alpha_3}{\alpha_1}\in\mathbb{R}\setminus\mathbb{Q}$ and $|\alpha|=1$. Then we have for all $\epsilon>0$
\begin{equation*}
\text{Var}\left(\frac{\mathcal{Z}}{\sqrt{m}}\right)\ll\frac{1}{m^{\frac{1}{7}-\epsilon}}.
\end{equation*}
\end{thm}
\noindent
Theorem \ref{resultirrat} is proven in Section \ref{sectpfirrat}.
\begin{thm}
\label{resulthalfrat}
Let $m\not\equiv 0,4,7 \pmod 8$ and the straight line segment $\mathcal{C}\subset\mathbb{T}^3$ be parametrised by $\gamma(t)=t(\alpha_1,\alpha_2,\alpha_3)$ with  $\frac{\alpha_2}{\alpha_1}\in\mathbb{Q}$, $\frac{\alpha_3}{\alpha_1}\in\mathbb{R}\setminus\mathbb{Q}$ and $|\alpha|=1$. Then we have for all $\epsilon>0$
\begin{equation*}
\text{Var}\left(\frac{\mathcal{Z}}{\sqrt{m}}\right)\ll\frac{1}{m^{\frac{1}{5}-\epsilon}}.
\end{equation*}
\end{thm}
\noindent
Theorem \ref{resulthalfrat} is proven in Section \ref{sectpfhalfrat}.

As a consequence of Theorems \ref{resultrat}, \ref{resultirrat} and \ref{resulthalfrat}, we may extend \eqref{distrasymean} to all straight lines. We may improve the bounds of Theorems \ref{resultirrat} and \ref{resulthalfrat} conditionally on a conjecture about lattice points in {\em spherical caps} (cf. Definition \ref{defcap}).
Jarnik \cite{jarnik} (see also \cite[Theorem 2.1]{brnoda}) proved that, for the sphere $R\mathcal{S}^2$, there is some $C>0$ such that all lattice points in a cap of radius $<CR^{\frac{1}{4}}$ lie on the same plane. By \eqref{kappa3bound}, it follows that every cap of radius $<CR^{\frac{1}{4}}$ contains $\ll R^\epsilon$ lattice points. Theorem \ref{resultcond} below is conditional on
\begin{conj}[Bourgain and Rudnick {\cite[Section 2.2]{brgafa}}]
\label{brgafaconj}
Let $\chi(R,s)$ be the maximal number of lattice points in a cap of radius $s$ of the sphere $R\mathcal{S}^2$. Then for all $\epsilon>0$ and $s<R^{1-\delta}$,
\begin{gather*}
\chi(R,s)
\ll
R^\epsilon\left(1+\frac{s^2}{R}\right)
\end{gather*}
as $R\to\infty$.
\end{conj}
\begin{thm}
\label{resultcond}
Assume Conjecture \ref{brgafaconj}. Let $m\not\equiv 0,4,7 \pmod 8$ and $\mathcal{C}$ be a straight line segment (rational or irrational) on $\mathbb{T}^3$.
Then we have for all $\epsilon>0$
\begin{equation*}
\text{Var}\left(\frac{\mathcal{Z}}{\sqrt{m}}\right)\ll\frac{1}{m^{\frac{1}{4}-\epsilon}}.
\end{equation*}
\end{thm}
\noindent
See Section \ref{conditional} for the proof of Theorem \ref{resultcond}.

In a previous paper \cite{maff2d}, we investigated nodal intersections against a straight line on the two-dimensional torus. For rational lines, Theorem \ref{resultrat} loses with respect to the two-dimensional analogue (cf. \cite[Theorem 1.1]{maff2d}). For irrational lines, Theorem \ref{resultirrat} prescribes an unconditional bound for all energies $m$,
whereas in the two-dimensional setting, an unconditional bound is only given for a {\em density one sequence} of energies, and a bound for all $m$ is given conditionally. These differences arise because the structure of lattice points on spheres differs significantly from that of lattice points on circles; see Section \ref{caps} for a more detailed discussion.

\subsection{Outline of the paper}
In Section \ref{krsection}, following the work of \cite{ruwiye}, we restrict the arithmetic random wave $F$ \eqref{arw} to the line $\mathcal{C}$, defining the {\em process} $f:=F(\gamma)$ indexed by the interval $[0,L]$ (see \eqref{effegen} below): the nodal intersections \eqref{Z} are counted by the zeros of $f$. The (factorial) moments of the number of zeros of a process are given, under certain hypotheses, by the {\em Kac-Rice formulas} (Theorem \ref{kacrice} below; also see \cite[\S 10]{cralea}). The expected number of zeros, for generic curves $\mathcal{C}$, was thus computed in \cite{ruwiye}. For the variance, however, the hypotheses of the Kac-Rice formula may fail in our setting; to treat this situation, we apply the {\em approximate Kac-Rice formula} of Rudnick, Wigman and Yesha \cite{ruwiye}, which bounds the variance using the second moment of the {\em covariance function} $r(t_1,t_2)=\mathbb{E}[F(\gamma(t_1))F(\gamma(t_2))]$ (see \eqref{rgen} below) and a couple of its derivatives.

Let us highlight the marked differences between the straight line and generic curve settings. If $\mathcal{C}$ is a straight line segment, the covariance function has the special form \eqref{r}, so that the process $f$ is {\em stationary}. This leads to a different method from \cite{ruwiye} of controlling the second moment of $r$, and specifically the quantity
\begin{equation}
\label{quantity}
\sum_{\substack
{\mu,\mu'\in\mathcal{E}\\\mu\neq\mu'}
}\left|\int_0^L e^{2\pi i\langle\mu-\mu',\gamma(t)\rangle}dt\right|^2
\end{equation}
with $\mathcal{E}$ as in \eqref{E}. Indeed, for curves with nowhere vanishing curvature, we have an oscillatory integral in \eqref{quantity}, thus Van der Corput's lemma applies (cf. \cite[Section 3]{ruwiye}) and reduces the problem to bounding the following summations over the lattice points:
\begin{equation*}
\sum_{\substack
{\mu,\mu'\in\mathcal{E}\\\mu\neq\mu'}
}\frac{1}{|\mu-\mu'|^j} \quad \text{for} \ j=2/3,1.
\end{equation*}
For straight lines, we may directly establish the following bound for the integral in \eqref{quantity}: if $\langle\mu-\mu',\alpha\rangle\neq 0$, then (cf. \eqref{min})
\begin{equation*}
\left|\int_0^L e^{2\pi it\langle\mu-\mu',\alpha\rangle}dt\right|^2
\ll
\min\left(1,\frac{1}{\langle\mu-\mu',\alpha\rangle^2}\right).
\end{equation*}
Thus, we need to understand the summation
\begin{equation}
\label{arithprob}
\sum_{
\substack
{\mu,\mu'\in\mathcal{E}\\\langle\mu-\mu',\alpha\rangle\neq 0}
}\frac{1}{\langle\mu-\mu',\alpha\rangle^2}
\end{equation}
where $\alpha$ is the direction of our straight line.

In Section \ref{rational}, we bound \eqref{arithprob} for $\alpha$ rational, and thus complete the proof of Theorem \ref{resultrat}. For $\alpha$ irrational, \eqref{arithprob} may be controlled by counting lattice points in certain regions of the sphere $R\mathcal{S}^2$. To this end, in Section \ref{caps}, we recall results about lattice points on spheres and in spherical caps. Moreover, in Sections \ref{segments} and \ref{segmentsbis} we prove bounds for the number of lattice points lying in regions of $R\mathcal{S}^2$ delimited by two parallel planes (i.e., {\em ``spherical segments"}; cf. Definition \ref{defseg}); some of these bounds rely on Diophantine approximation. Theorems \ref{resultirrat}, \ref{resulthalfrat} and \ref{resultcond} are thus established in Sections \ref{irrational}, \ref{halfrational} and \ref{conditional} respectively.

\section{A Kac-Rice type bound}
\label{krsection}
The arithmetic random wave \eqref{arw} is a centred real Gaussian random field on the torus. For a centred finite-variance random field $G:\mathcal{X}\to\mathbb{R}$ defined on a measurable space $\mathcal{X}$, its covariance function $r_G:\mathcal{X}\times \mathcal{X}\to\mathbb{R}$,
\begin{equation*}
r_G(x,y):=\mathbb{E}[G(x)\cdot G(y)]
\end{equation*}
is non-negative definite (see \cite[\S 5.1]{cralea}). Every centred real Gaussian random field $G$ is completely determined by $r_G$ (see Kolmogorov's Theorem \cite[\S 3.3]{cralea}).
\\
Moreover, the arithmetic random wave is a {\em stationary} random field, as its covariance function depends on the difference $x-y$ only:
\begin{equation*}
r_F(x,y)
=
\mathbb{E}[F(x)\cdot F(y)]
=
\frac{1}{N}\sum_{\mu\in\mathcal{E}} e^{2\pi i\langle\mu,(x-y)\rangle},
\qquad\quad
x,y\in\mathbb{T}^3.
\end{equation*}
For now we assume $\mathcal{C}$ to be a smooth toral curve (allowing but not imposing it to be a straight line segment), with arc-length parametrisation given by $\gamma(t): [0,L]\to\mathbb{T}^3$. 
We define the (in general non-stationary) process $f:[0,L]\to\mathbb{R}$,
\begin{equation}
\label{effegen}
f(t):=F(\gamma(t))=
\frac{1}{\sqrt{N}}
\sum_{\mu\in\mathcal{E}}
a_{\mu}
e^{2\pi i\langle\mu,\gamma(t)\rangle},
\end{equation}
which is the restriction of $F$ along $\mathcal{C}$. Its covariance function is
\begin{equation}
\label{rgen}
r(t_1,t_2)
=
\frac{1}{N}\sum_{\mu\in\mathcal{E}} e^{2\pi i\langle\mu,\gamma(t_1)-\gamma(t_2)\rangle}.
\end{equation}
The number of nodal intersections $\mathcal{Z}$ (recall \eqref{Z}) is thus given by the number of zeros of $f$. The (factorial) moments of a random variable that counts the number of zeros of a Gaussian process may be computed via the \textbf{Kac-Rice formulas} (see \cite[\S 10]{cralea}, and \cite[Theorem 3.2]{azawsc}). Let $X:I\to\mathbb{R}$ be a (a.s. $C^1$-smooth, say) Gaussian process on an interval $I\subseteq\mathbb{R}$. For $j\geq 1$ and distinct points $t_1,\dots,t_j\in I$, denote
\begin{equation*}
\phi_{t_1,\dots,t_j}
\end{equation*}
the probability density function of the Gaussian random vector
\begin{equation*}
(X(t_1),\dots X(t_j))\in\mathbb{R}^j.
\end{equation*}
For distinct points $t_1,\dots,t_j$, define the {\em $j$-th zero-intensity} of $X$ to be the conditional Gaussian expectation 
\begin{multline*}
K_j(t_1,\dots,t_j)\\=\phi_{t_1,\dots,t_j}(0,\dots,0)\cdot\mathbb{E}\left[|X'(t_1)\dots X'(t_j)|\big| X(t_1)=0,\dots,X(t_j)=0\right],
\end{multline*}
where $X'$ denotes the first derivative of $X$. We remark that $K_j$ admits a continuation to a smooth function on the whole of $I^j$ (cf. \cite[Section 3.1]{rudwig}).
\begin{thm}[{\cite[Theorem 3.2]{azawsc}; see also \cite[Theorem 2.1]{ruwiye}}]
\label{kacrice}
Let $X:I\to\mathbb{R}$ be a Gaussian process on an interval $I\subseteq\mathbb{R}$, having $C^1$ paths. Denote $\mathcal{A}$ the number of zeros of $X$ on $I$.
Let $j$ be a positive integer. Assume that for every $j$ {\em pairwise distinct} points $t_1,\dots t_j\in I$ the joint distribution of $(X(t_1),\dots X(t_j))\in\mathbb{R}^j$ is non-degenerate. Then
\begin{equation}
\label{kacrice2}
\mathbb{E}[\mathcal{A}^{[j]}]=
\int_{I^j}K_j(t_1,\dots t_j)dt_1\dots dt_j,
\end{equation}
where
\begin{equation*}
\mathcal{A}^{[j]}=
\begin{cases}
\mathcal{A}(\mathcal{A}-1)\cdots(\mathcal{A}-j+1) &\text{if}\ \mathcal{A}\geq j\geq 1
\\
0 &\text{otherwise.}
\end{cases}
\end{equation*}
\end{thm}
For the process $f$ as in \eqref{effegen}, the non-degeneracy condition of Theorem \ref{kacrice}  is automatically satisfied when $j=1$, since $f$ is unit variance: therefore,
\begin{equation}
\label{kacrice1}
\mathbb{E}[\mathcal{Z}]=
\int_{0}^{L}K_1(t)dt.
\end{equation}
Rudnick, Wigman and Yesha (\cite[Lemma 2.3]{ruwiye}) proved that, on the $d$-dimensional torus $\mathbb{T}^d$, $K_1(t)\equiv\frac{2}{\sqrt{d}}\sqrt{m}$, and hence by \eqref{kacrice1}, they computed the expected intersection number to be \eqref{expectation}.

For the nodal intersections variance, the non-degeneracy hypothesis of Theorem \ref{kacrice} is equivalent to the covariance function \eqref{rgen} of the process (which also verifies $|r|\leq 1$) satisfying $r(t_1,t_2)\neq\pm 1$ for all $t_1\neq t_2$:
this may fail for $f$ as in (\ref{effegen}). To resolve this situation, Rudnick, Wigman and Yesha \cite{ruwiye} developed an \textbf{approximate Kac-Rice formula}, thereby reducing the variance problem to bounding the second moment of the covariance function and a couple of its derivatives along $\mathcal{C}$, namely
\begin{equation*}
r_1:=\frac{\partial r(t_1,t_2)}{\partial t_1},
\qquad
r_2:=\frac{\partial r(t_1,t_2)}{\partial t_2}
\qquad
\text{and}
\quad
r_{12}:=\frac{\partial^2 r(t_1,t_2)}{\partial t_1\partial t_2}.
\end{equation*}
\begin{prop}[\emph{approximate Kac-Rice bound} {\cite[Proposition 2.2]{ruwiye}}]
\label{approxKR}
For smooth toral curves, we have
\begin{equation*}
\text{Var}(\mathcal{Z})
=
m
\cdot
O
(\mathcal{R}_2(m))
\end{equation*}
where
\begin{equation}
\label{2ndmom}
\mathcal{R}_2(m)
:=
\int_0^L \int_0^L 
\left[
r^2+\bigg(\frac{r_1}{\sqrt{m}}\bigg)^2+\bigg(\frac{r_2}{\sqrt{m}}\bigg)^2+\bigg(\frac{r_{12}}{m}\bigg)^2
\right]
dt_1dt_2
.
\end{equation}
\end{prop}
\noindent
In \cite[Proposition 1.3]{rudwig}, a precise asymptotic formula is given for the variance of
the nodal intersections number of arithmetic random waves on the two-dimensional torus
against a smooth curve with nowhere zero curvature. The bound of Proposition \ref{approxKR} is sufficient for the purpose of the present work.

From this point on, assume $\mathcal{C}\subset \mathbb{T}^3$ to be a straight line segment as in \eqref{C}. We may thus rewrite \eqref{effegen} as
\begin{equation}
\label{effe}
f(t)=
\frac{1}{\sqrt{N}}
\sum_{\mu\in\mathcal{E}}
a_{\mu}
e^{2\pi it\langle\mu,\alpha\rangle},
\end{equation}
and its covariance function \eqref{rgen} as
\begin{equation}
\label{r}
r(t_1,t_2)
=
\frac{1}{N}\sum_{\mu\in\mathcal{E}} e^{2\pi i(t_1-t_2)\langle\mu,\alpha\rangle}.
\end{equation}
The process \eqref{effe} is {\em stationary}: indeed, \eqref{r} depends on the difference $t_1-t_2$ only. Proposition \ref{approxKR} holds for all smooth curves $\mathcal{C}$, and in particular for straight line segments. We may further reduce our problem to bounding a sum over the lattice points (see Section \ref{auxpfs} for the proof of the following Lemma; cf. \cite[Lemma 6.1]{maff2d}).
\begin{lemma}
\label{rsq}
If $\mathcal{C}\subset \mathbb{T}^3$ is a straight line segment as in \eqref{C}, then we have
\begin{equation*}
\mathcal{R}_2(m)
\ll \frac{1}{N^2}
\sum_{(\mu,\mu')\in\mathcal{E}^2}
\left|\int_0^L e^{2\pi it\langle\mu-\mu',\alpha\rangle}dt\right|^2.
\end{equation*}
\end{lemma}
\begin{prop}
\label{mainprop}
If $\mathcal{C}\subset \mathbb{T}^3$ is a straight line segment as in \eqref{C}, then we have
\begin{equation}
\label{firststep}
\text{Var}\left(\frac{\mathcal{Z}}{\sqrt{m}}\right)
\ll
\frac{1}{N^2}
\sum_{(\mu,\mu')\in\mathcal{E}^2}
\left|\int_0^L e^{2\pi it\langle\mu-\mu',\alpha\rangle}dt\right|^2.
\end{equation}
\end{prop}
\begin{proof}[Proof of Proposition \ref{mainprop} assuming Lemma \ref{rsq}]
By Proposition \ref{approxKR} and Lemma \ref{rsq},
\begin{equation*}
\text{Var}\left(\frac{\mathcal{Z}}{\sqrt{m}}\right)
\ll
\mathcal{R}_2(m)
\ll
\frac{1}{N^2}
\sum_{(\mu,\mu')\in\mathcal{E}^2}
\left|\int_0^L e^{2\pi it\langle\mu-\mu',\alpha\rangle}dt\right|^2.
\end{equation*}
\end{proof}

\section{Rational lines: proof of Theorem \ref{resultrat}}
\label{rational}
Recall that $R=\sqrt{m}$, and the notation
\begin{equation*}
\mathcal{E}=\mathcal{E}(m):=\{\mu=(\mu_1,\mu_2,\mu_3)\in\mathbb{Z}^3 : \mu_1^2+\mu_2^2+\mu_3^2=m\}
\end{equation*}
for the lattice point set and
\begin{equation*}
N=N_m:=|\mathcal{E}|=r_3(m)
\end{equation*}
for its cardinality. Moreover, recall that $\kappa(R)$ denotes the maximal number of lattice points in the intersection of $R\mathcal{S}^{2}$ and a plane.
\begin{lemma}
\label{ratlem1}
For $\alpha\in\mathbb{R}^3$,
\begin{equation}
\label{scalprod0}
\#\{(\mu,\mu')\in\mathcal{E}^2: \ \langle\mu-\mu',\alpha\rangle=0\}
\leq
N\cdot \kappa(\sqrt{m})
.
\end{equation}
\end{lemma}
\begin{proof}
We rewrite the LHS of \eqref{scalprod0} as
\begin{equation*}
\sum_{\mu\in\mathcal{E}}
\#\{\mu': \ \langle\mu-\mu',\alpha\rangle=0\}
=
\sum_{\mu\in\mathcal{E}}
\#\{\mu': \ \langle\mu',\alpha\rangle=\langle\mu,\alpha\rangle\}
.
\end{equation*}
This means $\mu'$ belongs to the plane
\begin{equation}
\label{plane}
\langle\alpha,(x,y,z)\rangle=\xi,
\end{equation}
where $\xi:=\langle\mu,\alpha\rangle\in\mathbb{R}$. By Definition \ref{kappa}, \eqref{plane} has at most $\kappa(\sqrt{m})$ solutions $(x,y,z)\in\mathcal{E}$. Therefore,
\begin{equation*}
\sum_{\mu\in\mathcal{E}}
\#\{\mu': \ \langle\mu',\alpha\rangle=\langle\mu,\alpha\rangle\}
\leq\sum_{\mu\in\mathcal{E}}\kappa(\sqrt{m})
=
N\cdot \kappa(\sqrt{m})
.
\end{equation*}
\end{proof}

\begin{lemma}
\label{ratlem2}
For rational vectors $\alpha$,
\begin{equation*}
\sum_{\langle\mu-\mu',\alpha\rangle\neq 0}
\frac{1}{\langle\mu-\mu',\alpha\rangle^2}
\ll_\alpha
N\cdot \kappa(\sqrt{m})
.
\end{equation*}
\end{lemma}
\begin{proof}
Up to multiplication by a constant, $\alpha$ has integer components:
\begin{equation*}
(\alpha_1,\alpha_2,\alpha_3)
=
\alpha_1
\cdot
\left(1,\frac{\alpha_2}{\alpha_1},\frac{\alpha_3}{\alpha_1}\right)
=
\alpha_1
\cdot
\left(1,\frac{p}{q},\frac{r}{s}\right)
\end{equation*}
where $p,q,r,s\in\mathbb{Z}$ and $q,s\neq 0$. Then
\begin{equation*}
(\alpha_1,\alpha_2,\alpha_3)
=
\alpha_1
\cdot
\frac{1}{q}
\cdot
\frac{1}{s}
\cdot
(qs,ps,qr)
=
\frac{\alpha_1}{qs}
\cdot
(a,b,c)
\end{equation*}
with $a,b,c\in\mathbb{Z}$. Therefore,
\begin{gather*}
\sum_{\langle\mu-\mu',\alpha\rangle\neq 0}
\frac{1}{\langle\mu-\mu',\alpha\rangle^2}
=
\sum_{\langle\mu-\mu',\alpha\rangle\neq 0}
\frac{1}{(\frac{\alpha_1}{qs})^2\cdot\langle\mu-\mu',(a,b,c)\rangle^2}
\\
\ll_\alpha
\sum_{\langle\mu-\mu',\alpha\rangle\neq 0}
\frac{1}{\langle\mu-\mu',(a,b,c)\rangle^2}
=
\sum_{\mu}
\sum_{k\neq 0}
\sum_{\substack{\mu'\\\langle\mu-\mu',(a,b,c)\rangle=k}}
\frac{1}{k^2}
\\
=
\sum_{\mu}
\sum_{k\neq 0}
\frac{1}{k^2}
\cdot
\#
\{\mu':
\langle (a,b,c),\mu'\rangle=\xi=\xi(\mu,k)\in\mathbb{Z}\}.
\end{gather*}
As $\mu'$ belongs to the plane $ax+by+cz=\xi$, we have at most $\kappa(\sqrt{m})$ solutions. Therefore,
\begin{equation*}
\sum_{\langle\mu-\mu',\alpha\rangle\neq 0}
\frac{1}{\langle\mu-\mu',\alpha\rangle^2}\ll_\alpha\sum_{\mu}
\sum_{k\neq 0}
\frac{\kappa(\sqrt{m})}{k^2}
\ll
\kappa(\sqrt{m})\sum_{\mu}1
=
N\cdot \kappa(\sqrt{m})
.
\end{equation*}
\end{proof}

\begin{proof}[Proof of Theorem \ref{resultrat}]
By Proposition \ref{mainprop}, we have \eqref{firststep}. We remark that, if $\langle\mu-\mu',\alpha\rangle\neq 0$, then (cf. \cite[(6.7) and (6.8)]{maff2d})
\begin{equation}
\label{min}
\left|\int_0^L e^{2\pi it\langle\mu-\mu',\alpha\rangle}dt\right|^2
\ll
\min\left(1,\frac{1}{\langle\mu-\mu',\alpha\rangle^2}\right).
\end{equation}
We separate the summation on the RHS of \eqref{firststep} and apply \eqref{min}:
\begin{equation}
\begin{aligned}
\label{twosums}
\text{Var}\left(\frac{\mathcal{Z}}{\sqrt{m}}\right)
\ll
\frac{1}{N^2}
\left(
\sum_{\langle\mu-\mu',\alpha\rangle=0}1
+
\sum_{\langle\mu-\mu',\alpha\rangle\neq 0}
\left|\int_0^L e^{2\pi it\langle\mu-\mu',\alpha\rangle}dt\right|^2
\right)
\\
\ll
\frac{1}{N^2}
\left(
\#\{(\mu,\mu')\in\mathcal{E}^2: \ \langle\mu-\mu',\alpha\rangle=0\}
+
\sum_{\langle\mu-\mu',\alpha\rangle\neq 0}
\frac{1}{\langle\mu-\mu',\alpha\rangle^2}
\right)
.
\end{aligned}
\end{equation}
Both summands on the RHS of \eqref{twosums} are $\ll N\cdot \kappa(\sqrt{m})$, by Lemmas \ref{ratlem1} and \ref{ratlem2} respectively.
\end{proof}
As mentioned in the Introduction, Theorem \ref{resultrat} loses by the factor $\kappa(\sqrt{m})$ with respect to the 2-dimensional case (cf. \cite[Theorem 1.1]{maff2d}): for on the radius $\sqrt{m}$ circle, the maximal number of lattice points on the same hyperplane (line) is $\kappa_2(\sqrt{m})\leq 2$; on the radius $\sqrt{m}$ sphere, the maximal number of lattice points on the same hyperplane (plane) is $\kappa_3(\sqrt{m})\ll ({\sqrt{m}})^\epsilon$ (recall \eqref{kappa3bound}).

\section{Background on lattice points on spheres}
\label{caps}
We now turn to the case of intersections with irrational lines; we will need upper bounds for the number of lattice points in specific regions of the sphere $R\mathcal{S}^2=\sqrt{m}\mathcal{S}^2$. Recall the notation for the lattice point set $\mathcal{E}$ and number $N$.
\subsection{The total number of lattice points}
\label{seclp}
There is at least one lattice point if and only if $m$ is not of the form $4^l(8k+7)$ for $k$ and $l$ non-negative integers. We have the upper bound (see e.g. \cite[Section 1]{bosaru} and \cite[Section 4]{ruwiye})
\begin{equation*}
N
\ll
(\sqrt{m})^{1+\epsilon} \quad\text{for all} \  \epsilon>0.
\end{equation*}
The condition $m\not\equiv 0,4,7 \pmod 8$ is equivalent to the existence of \textit{primitive} lattice points $(\mu_1,\mu_2,\mu_3)$, meaning $\mu_1,\mu_2,\mu_3$ are coprime (\cite[Section 1]{bosaru}, \cite[Section 4]{ruwiye}). In this case, 
we have both lower and upper bounds
\begin{equation}
\label{totnumlp}
R^{1-\epsilon}
=
(\sqrt{m})^{1-\epsilon}
\ll
N
\ll
(\sqrt{m})^{1+\epsilon}
=
R^{1+\epsilon} \quad\text{for all} \  \epsilon>0.
\end{equation}
This lower bound is ineffective: the behaviour of $N=r_3(m)$ is not completely understood (\cite[Section 1]{bosaru}).
In the final paragraph of Section \ref{rational}, we noted a key difference between lattice points on spheres and on circles, namely, the upper bounds for the number of lattice points lying on a hyperplane. Another key difference between the two- and three-dimensional settings is the total number of lattice points. Recall the notation $_2\mathcal{E}(m)$ \eqref{lpsetd} for the set of all lattice points on the circle of radius $\sqrt{m}$. Their cardinality, i.e. the number ${_2N_m}=r_2(m)$ of ways that $m$ may be written as a sum of two squares, has the upper bound (see \cite[\S 18.7]{harwri})
\begin{equation}
\label{lpcircles}
_2N_m\ll m^\epsilon \quad\text{for all} \  \epsilon>0,
\end{equation}
very different from \eqref{totnumlp}.

\subsection{Lattice points in spherical caps}
\begin{defin}
\label{defcap}
Given a sphere $\Sigma$ in $\mathbb{R}^3$, with centre $O$ and radius $R$, and a point $P\in\Sigma$, we define the \textbf{spherical cap} $T$ centred at $P$ to be the intersection of $\Sigma$ with the ball $B_s(P)$ of radius $s$ centred at $P$. We will call $s$ the \textbf{radius of the cap}, and the unit vector $\beta:=\frac{\overrightarrow{OP}}{R}$ the  \textbf{direction} of $T$.
\\
The intersection of $\Sigma$ with the boundary of $B_s(P)$ is a circle; it will be called the \textbf{base} of $T$, and the {\bf radius of the base} will be denoted $k$. Let $Q,Q'$ be two points on the base which are diametrically opposite (note $\overline{PQ}=\overline{PQ'}=s$): we define the \textbf{opening angle} of $T$ to be $\theta=\widehat{QOQ'}$. The \textbf{height} $h$ of $T$ is the distance between the point $P$ and the base.
\\
Equivalently, $T$ may be defined as the region of the sphere $\Sigma$ delimited by a plane; the intersection of this plane with $\Sigma$ is the base of $T$.
\end{defin}
If $s$, $h$, $k$ and $\theta$ denote the radius, height, radius of the base, and opening angle 
of $T$ respectively, then we have $0\leq s\leq 2R$, $0\leq h\leq 2R$, $0\leq k\leq R$ and $0\leq \theta\leq \pi$. Furthermore, geometric considerations give the relations
\begin{equation}
\label{shR}
k^2+h^2=s^2=2Rh
\end{equation}
and
\begin{equation}
\label{sRtheta}
s=2R\sin\left(\theta/4\right).
\end{equation}
From \eqref{shR} and \eqref{sRtheta} we deduce
\begin{equation}
\label{htheta}
\theta=4\arcsin\left(\sqrt{h/2R}\right).
\end{equation}

We shall denote
\begin{equation}
\label{chi(R,s)}
\chi(R,s)
=\max_T\#\{\mu\in\mathbb{Z}^3\cap T\}
\end{equation}
the maximal number of lattice points belonging to a spherical cap $T\subset R\mathcal{S}^2$ of radius $s$: \eqref{chi(R,s)} is a $3$-dimensional analogue of lattice points on short arcs of a circle.
\begin{lemma}[Bourgain and Rudnick {\cite[Lemma 2.1]{brgafa}}]
\label{lemma2.1}
We have for all $\epsilon>0$,
\begin{equation*}
\chi(R,s)
\ll
R^\epsilon\left(1+\frac{s^2}{R^\frac{1}{2}}\right)
\end{equation*}
as $R\to\infty$.
\end{lemma}
\noindent
Compare this result with Conjecture \ref{brgafaconj}.


\section{Spherical segments}
\label{segments}
\subsection{Definitions and notation}
\begin{defin}
\label{defseg}
Given a sphere $\Sigma$ in $\mathbb{R}^3$, and two parallel planes $\Pi_1,\Pi_2$ which both have non-empty intersection with $\Sigma$, we call \textbf{spherical segment} $S$ the region of the sphere delimited by $\Pi_1,\Pi_2$. The two \textbf{bases} of $S$ are the circles $\mathcal{B}_1=\Sigma\cap\Pi_1$ and $\mathcal{B}_2=\Sigma\cap\Pi_2$. We always assume that $\mathcal{B}_2$ is the larger of the two bases.
\end{defin}
It will be convenient to always assume a spherical segment $S$ to be contained in a hemisphere. If this is not the case, 
then there exist two spherical segments $S_1$ and $S_2$, each contained in a hemisphere, such that $S_1\cup S_2=S$, $S_1\cap S_2=\mathcal{B}$ with $\mathcal{B}$ a great circle of the sphere. Therefore, a property of $S$ may be derived by working on $S_1$ and $S_2$.

\begin{defin}
\label{defseg2}
Given a spherical segment $S$ with same notation as in Definition \ref{defseg}, we define its \textbf{height} $h$ to be the distance between $\Pi_1$ and $\Pi_2$. We will denote
$k$ the \textbf{radius of the larger base} $\mathcal{B}_2$. Moreover, let $\Gamma$ be a great circle of the sphere $\Sigma$, lying on a plane perpendicular to $\Pi_1$ and $\Pi_2$. Denote $\{A,B\}:=\mathcal{B}_1\cap\Gamma$, $\{C,D\}:=\mathcal{B}_2\cap\Gamma$
and call $O$ the centre of the sphere. We define the \textbf{opening angle} of $S$ to be $\theta=\widehat{AOC}+\widehat{BOD}=2\cdot\widehat{AOC}$.
\end{defin}
Consider the special case when the spherical segment is a cap, i.e. $\mathcal{B}_1$ is a point. With the notation of Definition \ref{defseg2}, since the points $A$ and $B$ coincide, we get $\theta=\widehat{AOC}+\widehat{BOD}=\widehat{COD}$, which is consistent with the definition of the opening angle for a spherical cap (cf. Definition \ref{defcap}). Note that any two of $h,k,\theta$ completely determine $S$ (recall we are assuming the segment to be contained is a hemisphere).
We always have $0\leq h\leq R$, $0\leq k\leq R$ and $0\leq \theta\leq \pi$.
\\
We may also regard a spherical segment $S$ as the difference set of two spherical caps $T_1$ and $T_2$:
\begin{equation*}
S=T_2\setminus T_1.
\end{equation*}
We will need the following lemma later; see Section \ref{auxpfs} for the proof.
\begin{lemma}
\label{kthetah}
Given a spherical segment $S\subset R\mathcal{S}^2$ of height $h(R)$, radius of larger base $k(R)$ and opening angle $\theta(R)$, as $R\to\infty$ we have
\begin{equation*}
k\theta\ll h.
\end{equation*}
\end{lemma}

\subsection{Lattice points in spherical segments: covering the segment with caps}
\label{segments2}
We want to give an upper bound for the maximal number of lattice points belonging to a spherical segment $S$ of the sphere $R\mathcal{S}^2$,
\begin{equation}
\label{psi}
\psi=\psi(R,h,k,\theta):=\max_S\#\{\mu\in\mathbb{Z}^3\cap S\},
\end{equation}
with $h,k,\theta$ as in Definition \ref{defseg2}.
\begin{prop}
\label{covercapsnew}
Let $S\subset R\mathcal{S}^2$ be a spherical segment of opening angle $\theta$ and radius of larger base $k$. Then for every real number $0<\Omega<R$,
\begin{equation}
\label{Omega}
\psi\leq
\chi(R,(2\pi+1/2)\Omega)
\cdot
\left\lceil\frac{k}{\Omega}\right\rceil
\cdot
\left\lceil \frac{R\theta}{\Omega}\right\rceil
\end{equation}
with $\chi(R,\cdot)$ as in \eqref{chi(R,s)}.
\end{prop}
\begin{proof}
Given a real number $0<\Omega<R$, we will partition $S$ into regions $\mathcal{R}_{ij}$ (described below), and then cover each $\mathcal{R}_{ij}$ with a spherical cap of radius $(2\pi+1/2)\Omega$. Therefore, $\psi$ does not exceed the number of lattice points $\chi(R,(2\pi+1/2)\Omega)$ in a cap, times the number of caps.
\\
The partitioning is done as follows. Denote $\mathcal{B}_1,\mathcal{B}_2$ the two bases of $S$, lying on the parallel planes $\Pi_1,\Pi_2$ respectively; the larger base $\mathcal{B}_2$ has radius $k$. Consider a set of great semicircles
\begin{equation*}
\left\{\Gamma_i, 1\leq i\leq \left\lceil\frac{k}{\Omega}\right\rceil\right\}
\end{equation*}
lying on planes all perpendicular to $\Pi_1,\Pi_2$, and chosen so that they partition the circle $\mathcal{B}_2$ into $\left\lceil\frac{k}{\Omega}\right\rceil$ identical arcs 
each of length
\begin{equation*}
\delta:=
\frac{2\pi k}{\left\lceil\frac{k}{\Omega}\right\rceil}\leq 2\pi \Omega.
\end{equation*}
For $1\leq i\leq\left\lceil\frac{k}{\Omega}\right\rceil$, 
the arcs $S\cap \Gamma_i$ have length $\frac{R\theta}{2}$. Moreover, let
\begin{equation*}
\left\{\Lambda_j, 1\leq j\leq \left\lfloor \frac{R\theta}{\Omega}\right\rfloor\right\}
\end{equation*}
be a set of circles on $R\mathcal{S}^2$, all lying on planes parallel to $\Pi_1,\Pi_2$, that partition each arc $S\cap \Gamma_i$ into $\left\lceil \frac{R\theta}{\Omega}\right\rceil$ identical smaller arcs of length
\begin{equation*}
\eta:=
\frac{R\theta}{2\left\lceil \frac{R\theta}{\Omega}\right\rceil}\leq\frac{1}{2} \Omega.
\end{equation*}
Notice that the $\Gamma_i$'s and $\Lambda_j$'s partition $S$ into
\begin{equation}
\label{numcapsnew}
\left\lceil\frac{k}{\Omega}\right\rceil\cdot\left\lceil \frac{R\theta}{\Omega}\right\rceil
\end{equation}
regions $\mathcal{R}_{ij}\subset R\mathcal{S}^2$. We now show that each $\mathcal{R}_{ij}$ may be covered by a spherical cap of radius $(2\pi+1/2)\Omega$.
\\
We will use the notation $\overset{\frown\Lambda}{AB}$ for an arc of a circle $\Lambda$ of the sphere $R\mathcal{S}^2$, of endpoints $A$ and $B$. The arc $\overset{\frown\Lambda}{AB}$ is a geodesic if and only if $\Lambda$ is a great circle of $R\mathcal{S}^2$. In this case, we will simply write $\overset{\frown}{AB}$. The region $\mathcal{R}_{ij}$ is delimited by four arcs:
\begin{equation*}
\overset{\displaystyle{\quad \frown_{\Lambda_j}}}{AB}\subset\Lambda_j,\quad\overset{\displaystyle{\frown}}{BC}\subset\Gamma_i,\quad\overset{\quad\displaystyle{\frown_{\Lambda_{j+1}}}}{CD}\subset\Lambda_{j+1},\quad\overset{\displaystyle{\frown}}{AD}\subset\Gamma_{i+1}.
\end{equation*}
By the construction of the circles $\{\Gamma_i\}_i$ and $\{\Lambda_j\}_j$, we have the relations
\begin{equation*}
\overset{\displaystyle{\quad \frown_{\Lambda_j}}}{AB}
<
\overset{\quad\displaystyle{\frown_{\Lambda_{j+1}}}}{CD}
\leq\delta,\quad
\overset{\displaystyle{\frown}}{BC}
=
\overset{\displaystyle{\frown}}{AD}
=\eta.
\end{equation*}
Given any point $P\in\mathcal{R}_{ij}$, we denote $\overline{AP}$ the euclidean distance between $A$ and $P$. Let us show that $\overline{AP}\leq(2\pi+1/2)\Omega$, so that $\mathcal{R}_{ij}$ may be covered by the spherical cap of radius $(2\pi+1/2)\Omega$ centred at $A$. Let $\Lambda_P$ be the circle on $R\mathcal{S}^2$ containing $P$ and lying on a plane parallel to $\Pi_1,\Pi_2$. Let $Q$ be the intersection between $\Lambda_P$ and $\overset{\frown}{AD}$. The euclidean distance between $A$ and $P$ is less than the length of the geodesic $\overset{\frown}{AP}$, which, in turn, is less that the sum of the lengths of the geodesics $\overset{\frown}{AQ}$ and $\overset{\frown}{QP}$. Moreover, we have
\begin{equation*}
\overset{\displaystyle{\frown}}{QP}\leq\overset{\quad\displaystyle{\frown_{\Lambda_P}}}{QP}\leq\overset{\quad\displaystyle{\frown_{\Lambda_{j+1}}}}{CD}\leq\delta\leq 2\pi \Omega
\qquad\text{and}\qquad
\overset{\displaystyle{\frown}}{AQ}\leq\overset{\displaystyle{\frown}}{AD}=\eta\leq\frac{1}{2} \Omega
.
\end{equation*}
It follows that, as desired,
\begin{equation*}
\overline{AP}<\overset{\displaystyle{\frown}}{AP}\leq\overset{\displaystyle{\frown}}{AQ}+\overset{\displaystyle{\frown}}{QP}\leq \left(2\pi+\frac{1}{2}\right)\Omega.
\end{equation*}
The total number of caps equals the number of regions \eqref{numcapsnew}; therefore,
\begin{equation*}
\psi
\leq 
\chi(R,(2\pi+1/2)\Omega)
\cdot
\left\lceil\frac{k}{\Omega}\right\rceil
\cdot
\left\lceil \frac{R\theta}{\Omega}\right\rceil
.
\end{equation*}
\end{proof}

\begin{cor}
\label{covercapscor1}
Let $S\subset R\mathcal{S}^2$ be a spherical segment of opening angle $\theta$ and radius of larger base $k$. Then for every real number $0<\Omega<R$ and for every $\epsilon>0$, we have unconditionally
\begin{equation}
\label{cor1}
\psi\ll R^{\epsilon}
\left(1+\frac{\Omega^2}{R^{\frac{1}{2}}}\right)
\cdot
\left\lceil\frac{k}{\Omega}\right\rceil
\cdot
\left\lceil \frac{R\theta}{\Omega}\right\rceil.
\end{equation}
\end{cor}
\begin{proof}
By Lemma \ref{lemma2.1}, we may unconditionally insert the bound
\begin{equation*}
\chi(R,(2\pi+1/2)\Omega)
\ll
R^\epsilon\left(1+\frac{((2\pi+1/2)\Omega)^2}{R^\frac{1}{2}}\right)
\ll
R^\epsilon\left(1+\frac{\Omega^2}{R^\frac{1}{2}}\right)
\end{equation*}
into \eqref{Omega}, obtaining \eqref{cor1}.
\end{proof}

\begin{cor}
\label{covercapscor2}
Assume Conjecture \ref{brgafaconj}. Let $S\subset R\mathcal{S}^2$ be a spherical segment of height $h$ and radius of larger base $k$. Then for every $\epsilon>0$,
\begin{equation*}
\psi\ll R^{\epsilon}
\cdot
(R^{1/2}+h).
\end{equation*}
\end{cor}
\begin{proof}
The opening angle 
of the spherical segment $S$ shall be denoted $\theta$.
By Proposition \ref{covercapsnew}, we have \eqref{Omega} for every real number $0<\Omega<R$.
By Conjecture \ref{brgafaconj}, it follows that, for every $\epsilon>0$,
\begin{equation*}
\psi\ll R^{\epsilon}
\cdot
\left(1+\frac{((2\pi+1/2)\Omega)^2}{R}\right)
\cdot
\left(1+\frac{k}{\Omega}\right)
\cdot
\left(1+\frac{R\theta}{\Omega}\right).
\end{equation*}
We take $\Omega=R^{1/2}$:
\begin{equation*}
\psi\ll R^{\epsilon}
\cdot
\left(1+\frac{k}{R^{1/2}}\right)
\cdot
(1+R^{1/2}\theta)
=
R^{\epsilon}
\left(1+\frac{k}{R^{1/2}}+R^{1/2}\theta+k\theta\right)
.
\end{equation*}
Since $0\leq k\leq R$ and $0\leq \theta\leq \pi$,
\begin{equation*}
\psi\ll R^{\epsilon}
(R^{1/2}+k\theta)
.
\end{equation*}
Finally, by Lemma \ref{kthetah}, we have $k\theta\ll h$.
\end{proof}

\section{Lattice points in spherical segments: Diophantine approximation}
\label{segmentsbis}
Recall the notation $\psi$ \eqref{psi} for the maximal number of lattice points lying on a spherical segment $S\subset R\mathcal{S}^2$ of height $h$, radius of larger base $k$, and opening angle $\theta$. The goal of this section is to prove a bound for $\psi$ which depends only on $\theta$. Recall Definition \ref{kappa} for $\kappa(R)$, and Definition \ref{defcap} for the direction $\beta$ of a spherical cap.
\begin{defin}
The \textbf{direction} of a spherical segment $S$ is the unit vector $\beta=(\beta_1,\beta_2,\beta_3)$ which is the direction of the two spherical caps $T_1,T_2$ satisfying
\begin{equation*}
S=T_2\setminus T_1.
\end{equation*}
\end{defin}
\begin{prop}
\label{lpseg}
Let $S\subset R\mathcal{S}^2$ be a spherical segment of opening angle $\theta$, radius of larger base $k$, and direction $\beta$, with $\frac{\beta_2}{\beta_1},\frac{\beta_3}{\beta_1}\in\mathbb{R}\setminus\mathbb{Q}$. Then the number of lattice points lying on $S$ satisfies
\begin{equation*}
\psi\ll \kappa(R)(1+R\cdot\theta^{1/3})
\end{equation*}
for $\theta\to 0$, the implied constant being absolute.
\end{prop}
\noindent
The proof of this result will be given at the end of the present section, following some preparation; we will apply the ideas of \cite[Lemma 2.3]{brgafa}. Firstly, we shall consider a spherical cap $T$ or segment $S'$, containing $S$, and of direction a rational vector $\frac{a}{|a|}$, where $a_1,a_2,a_3$ are parameters. Thus
\begin{equation*}
\psi\leq \#\{\text{lattice points in } T \text{ or } S'\}.
\end{equation*}
We will then have to work with a larger portion of the sphere; however, as the new cap or segment's direction is a rational vector, the `slicing' method of \cite{brgafa} may be applied; thus we will show
\begin{equation*}
\psi\ll\kappa(R)\cdot[1+R|a|(\theta+\varphi)]
\end{equation*}
where $a=(a_1,a_2,a_3)\in\mathbb{Z}^3$, $\varphi$ is the angle between $\beta$ and $a$. Finally, to minimise the quantity $|a|(\theta+\varphi)$, we will choose values for the parameters $a_1,a_2,a_3\in\mathbb{Z}$ such that both $|a|$ and $\varphi$ are small, applying Diophantine approximation. Let us commence this preparatory work.

To bound the number of lattice points in a spherical segment of direction a {\em rational vector}, we apply the `slicing' method of \cite[proof of Lemma 2.3]{brgafa}; see also Yesha \cite[Lemma A.1]{yesh13}.
\begin{prop}
\label{ratsect}
Let $S\subset R\mathcal{S}^2$ be a spherical segment of height $h$, radius of larger base $k$, and direction a rational vector $\frac{b}{|b|}$, where $b\in\mathbb{Z}^3$. Then, for any $0\leq h\leq R$,
\begin{equation}
\label{slicing}
\psi\leq \kappa(R)\cdot(1+|(b_1,b_2,b_3)|\cdot h).
\end{equation}
In particular, $\forall\epsilon>0$,
\begin{equation}
\label{slicingcor}
\psi\ll_{b} R^\epsilon\cdot(1+h).
\end{equation}
\end{prop}
\begin{proof}
Since $b\in\mathbb{Z}^3$, then for all lattice points $\mu$, we have $\langle b,\mu\rangle=n\in\mathbb{Z}$, hence each lattice point on $S$ belongs to a plane
\begin{equation}
\label{planed}
\langle(b_1,b_2,b_3),(x,y,z)\rangle=n
\end{equation}
intersecting $S$.
It follows that $\psi$ is bounded by the number $\nu(h,b)$ of planes \eqref{planed} intersecting $S$ times the number of lattice points lying on each plane. Therefore, recalling Definition \ref{kappa}, we have
\begin{equation}
\label{nu}
\psi\leq \nu(h,b)\cdot\kappa(R).
\end{equation}
It remains to bound $\nu(h,b)$. We claim that the minimal distance between two adjacent planes \eqref{planed} both containing at least one lattice point is $\frac{n'}{|(b_1,b_2,b_3)|}$, $n'$ being a positive integer. Indeed, consider two planes
\begin{equation*}
\langle(b_1,b_2,b_3),(x,y,z)\rangle=n \qquad\text{ and }\qquad \langle(b_1,b_2,b_3),(x,y,z)\rangle=n+n',
\end{equation*}
each containing at least one lattice point, with $n'$ positive and as small as possible. Fix any point $P$ on the former of these two planes, and a point $Q$ on the latter so that the line through $P,Q$ is orthogonal to the planes. The sought distance is thus $|Q-P|$. We have
\begin{numcases}{}
\label{systeq1}
b_1x_P+b_2y_P+b_3z_P=n
\\
\label{systeq2}
b_1x_Q+b_2y_Q+b_3z_Q=n+n'
\\
\label{systeq3}
Q=P+\lambda (b_1,b_2,b_3),
\end{numcases}
which yields $|Q-P|=|\lambda\cdot(b_1,b_2,b_3)|$, with $\lambda$ to be determined. By subtracting \eqref{systeq1} from \eqref{systeq2}:
\begin{equation*}
b_1(x_Q-x_P)+b_2(y_Q-y_P)+b_3(z_Q-z_P)=n'
\end{equation*}
i.e.,
\begin{equation}
\label{system}
\langle(b_1,b_2,b_3),Q-P\rangle=n'.
\end{equation}
Inserting \eqref{systeq3} into \eqref{system} yields 
\begin{gather*}
\langle(b_1,b_2,b_3),\lambda (b_1,b_2,b_3)\rangle=n'
\Rightarrow
\lambda\cdot|(b_1,b_2,b_3)|^2=n'
\Rightarrow
\lambda=\frac{n'}{|(b_1,b_2,b_3)|^2} 
\\
\Rightarrow
|Q-P|=|\lambda(b_1,b_2,b_3)|=\frac{n'}{|(b_1,b_2,b_3)|}.
\end{gather*}
As the height of the segment is $h$, we get
\begin{equation*}
\nu(h,b)\leq 1+\frac{h}{|Q-P|}=1+|(b_1,b_2,b_3)|\cdot \frac{h}{n'}.
\end{equation*}
Since $n'\geq 1$, it follows that 
\begin{equation*}
\nu(h,b)\leq 1+|(b_1,b_2,b_3)|\cdot h
\end{equation*}
which together with \eqref{nu} implies \eqref{slicing}. In particular, recalling \eqref{kappa3bound}, we get \eqref{slicingcor}.
\end{proof}
\noindent
The proof of the following lemma may be found in Section \ref{auxpfs}.
\begin{lemma}
\label{lemmaclaim}
Let $S\subset R\mathcal{S}^2$ be a spherical segment of opening angle $\theta$, radius of larger base $k$, and direction the unit vector $\beta$. For every non-zero $a=(a_1,a_2,a_3)\in\mathbb{Z}^3$, the maximal number of lattice points lying on $S$ satisfies
\begin{equation}
\label{claim}
\psi\ll \kappa(R)\cdot[1+R|a|(\theta+\varphi)]
\end{equation}
where $\varphi$ is the angle between $\beta$ and $a$, and the implied constant is absolute.
\end{lemma}
Next, we state Dirichlet's theorem on simultaneous approximation (see \cite[proof of Lemma 2.5]{brgafa}; see also \cite[\S 11.12]{harwri}, or \cite[section II, Theorem 1A]{daschm}).
\begin{prop}[Dirichlet]
\label{dirichsim}
Given $\zeta_1,\zeta_2\in\mathbb{R}\setminus\mathbb{Q}$ and an integer $H\geq 1$, there exist $q,p_1,p_2\in\mathbb{Z}$ so that $1\leq q\leq H^2$ and
\begin{equation*}
\left|\zeta_1-\frac{p_1}{q}\right|,
\left|\zeta_2-\frac{p_2}{q}\right|
<
\frac{1}{q H}.
\end{equation*}
\end{prop}
\begin{lemma}
\label{trineq}
Let $v,w$ be two non-zero vectors of $\mathbb{R}^n$. Then
\begin{equation*}
\left|\frac{v}{|v|}-\frac{w}{|w|}\right|
\leq
2\frac{|v-w|}{|w|}.
\end{equation*}
\end{lemma}
\noindent
The proof of Lemma \ref{trineq} is an application of the triangle inequality and is deferred to Section \ref{auxpfs}.

\begin{lemma}
\label{diophapproxlemma1}
For all vectors $\alpha\in\mathbb{R}^3$ with $\frac{\alpha_2}{\alpha_1},\frac{\alpha_3}{\alpha_1}\in\mathbb{R}\setminus\mathbb{Q}$ and $|\alpha|=1$, and for all integers $H\geq 1$, there exists $a\in\mathbb{Z}^3$ satisfying
\begin{numcases}{}
\label{syst2cond1}
|a|\leq 3H^2
\\
\label{syst2cond2}
\left|\alpha-\frac{a}{|a|}\right|
<
\frac{6\sqrt{2}}{|a|H}.
\end{numcases}
\end{lemma}
\begin{proof}[Proof of Lemma \ref{diophapproxlemma1} assuming Lemma \ref{trineq}]
As in \cite[proof of Lemma 2.3]{brgafa}, assume $|\alpha_1|=\max (|\alpha_1|,|\alpha_2|,|\alpha_3|)$. Take $\zeta_1=\frac{\alpha_2}{\alpha_1}$, $\zeta_2=\frac{\alpha_3}{\alpha_1}$ and a large integer $H$: by Proposition \ref{dirichsim}, there exist integers $q,p_1,p_2$ so that $1\leq q\leq H^2$ and
\begin{equation*}
\left|\frac{\alpha_2}{\alpha_1}-\frac{p_1}{q}\right|,
\left|\frac{\alpha_3}{\alpha_1}-\frac{p_2}{q}\right|
<
\frac{1}{q H}.
\end{equation*}
We may assume $\alpha_1>0$ (in case $\alpha_1<0$, take $-\alpha$), and set $a=(a_1,a_2,a_3):=(q,p_1,p_2)\in\mathbb{Z}^3$. Then
\begin{gather*}
|\alpha_1|=\max (|\alpha_1|,|\alpha_2|,|\alpha_3|)
\Rightarrow
0\leq \left|\frac{\alpha_2}{\alpha_1}\right|,
\left|\frac{\alpha_3}{\alpha_1}\right|
\leq 1
\\
\Rightarrow
0\leq \left|\frac{p_1}{q}\right|,
\left|\frac{p_2}{q}\right|
\leq 1+\frac{1}{qH}\leq 2
\Rightarrow
|p_1|,|p_2|\leq 2q
\\
\Rightarrow
|a|^2=q^2+p_1^2+p_2^2\leq q^2+4q^2+4q^2=9q^2
\Rightarrow
|a|\leq 3q\leq 3H^2,
\end{gather*}
and \eqref{syst2cond1} is satisfied. We now turn to proving \eqref{syst2cond2}. We define the vector $d:=\frac{\alpha_1}{q}\cdot a\in\mathbb{R}^3$, hence (as $\alpha_1>0$)
\begin{equation*}
\frac{d}{|d|}=
\frac
{\frac{\alpha_1}{q}\cdot a}
{\frac{|\alpha_1|}{|q|}\cdot |a|}
=
\frac{a}{|a|}.
\end{equation*}
We apply Lemma \ref{trineq} with $w=\alpha$ and $v=d$, recalling that $|\alpha|=1$:
\begin{equation}
\label{fact1}
\left|\alpha-\frac{a}{|a|}\right|
=
\left|\frac{\alpha}{|\alpha|}-\frac{d}{|d|}\right|
\leq
2\frac{|\alpha-d|}{|\alpha|}
=
2|\alpha-d|.
\end{equation}
Moreover,
\begin{equation}
\begin{aligned}
\label{fact2}
|\alpha-d|
=
\left|
\alpha-\frac{\alpha_1}{q}\cdot (q,p_1,p_2)
\right|
=
\left|
\left(
\alpha_1-\frac{\alpha_1}{q}\cdot q,
\alpha_2-\frac{\alpha_1}{q}\cdot p_1,
\alpha_3-\frac{\alpha_1}{q}\cdot p_2
\right)
\right|
\\
=
|\alpha_1|
\cdot
\left|
\left(
0,
\frac{\alpha_2}{\alpha_1}-\frac{p_1}{q},
\frac{\alpha_3}{\alpha_1}-\frac{p_2}{q}
\right)
\right|
=
|\alpha_1|
\cdot
\bigg(
\left|
\frac{\alpha_2}{\alpha_1}-\frac{p_1}{q}
\right|^2
+
\left|
\frac{\alpha_3}{\alpha_1}-\frac{p_2}{q}
\right|^2
\bigg)^{1/2}
\\
<
|\alpha_1|
\cdot
\bigg(2\left(\frac{1}{q H}\right)^2\bigg)^{1/2}
=
\sqrt{2}\cdot|\alpha_1|\cdot\frac{1}{q H}
<
\sqrt{2}\cdot\frac{1}{q H}.
\end{aligned}
\end{equation}
Since $|a|\leq 3q$, we have
\begin{equation}
\label{fact3}
\frac{1}{q}\leq\frac{3}{|a|}.
\end{equation}
Combining \eqref{fact1}, \eqref{fact2} and \eqref{fact3},
\begin{equation*}
\left|\alpha-\frac{a}{|a|}\right|
\leq
2|\alpha-d|
<
2\sqrt{2}\cdot\frac{1}{q H}
\leq
6\sqrt{2}\cdot\frac{1}{|a|H}
\end{equation*}
and \eqref{syst2cond2} is satisfied.
\end{proof}


\begin{proof}[Proof of Proposition \ref{lpseg} assuming the preparatory results]
By Lemma \ref{lemmaclaim}, we have for every non-zero $a=(a_1,a_2,a_3)\in\mathbb{Z}^3$,
\begin{equation}
\label{applylemmaclaim}
\psi\ll\kappa(R)\cdot[1+R|a|(\theta+\varphi)]
\end{equation}
where $\varphi$ is the angle between $\beta$ and $a$. We are then looking for $a\in\mathbb{Z}^3$ which minimises the quantity
$
|a|(\theta+\varphi)
$.
We claim that
\begin{equation}
\label{phixclaim}
\varphi\sim\left|\beta-\frac{a}{|a|}\right| \quad\text{as }\theta\to 0
\end{equation}
(this will be shown at the end of the proof). By \eqref{applylemmaclaim} and \eqref{phixclaim},
\begin{equation}
\label{mezzo}
\psi\ll \kappa(R)\cdot\left[1+R\left(|a|\theta+|a|\left|\beta-\frac{a}{|a|}\right|\right)\right].
\end{equation}
It then suffices to bound
\begin{equation*}
|a|\theta+|a|\left|\beta-\frac{a}{|a|}\right|.
\end{equation*}
We want $a=(a_1,a_2,a_3)$ s.t. $|a|$ and $|\beta-\frac{a}{|a|}|$ are both small. We apply Lemma \ref{diophapproxlemma1} with $\alpha=\beta$: for all integers $H\geq 1$, there exists $a=(a_1,a_2,a_3)$ so that 
\begin{equation*}
|a|\theta+|a|\left|\beta-\frac{a}{|a|}\right|
<
3H^2\cdot\theta+|a|\frac{6\sqrt{2}}{|a|H}
=
3H^2\cdot\theta+\frac{6\sqrt{2}}{H}
.
\end{equation*}
The tradeoff gives us the choice $H=\big\lfloor\big(\frac{2\sqrt{2}}{\theta}\big)^{1/3}\big\rfloor=\Big\lfloor\frac{\sqrt{2}}{\theta^{1/3}}\Big\rfloor$, hence
\begin{equation*}
|a|\theta+|a|\left|\beta-\frac{a}{|a|}\right|
<
3
\left(
2\theta^{1/3}
+
2\sqrt{2}\frac{1}{\lfloor\sqrt{2}/\theta^{1/3}\rfloor}
\right)
\ll
\theta^{1/3}.
\end{equation*}
Inserting this bound into \eqref{mezzo} yields the statement of Proposition \ref{lpseg}:
\begin{equation*}
\psi
\ll
\kappa(R)\cdot[1+R\cdot\theta^{1/3}]
.
\end{equation*}
It remains to show \eqref{phixclaim}. Consider the triangle of sides $\beta,\frac{a}{|a|}$ and $\beta-\frac{a}{|a|}$, of lengths $|\beta|=1, |\frac{a}{|a|}|=1$, and $x:=|\beta-\frac{a}{|a|}|$ respectively. The angle opposite the side of length $x$ is $\varphi$, hence 
$x=2\sin(\frac{\varphi}{2})$. If we show that $\varphi\to 0$ as $\theta\to 0$, it will imply $x=2\sin(\frac{\varphi}{2})\sim 2\cdot\frac{\varphi}{2}=\varphi$; it will suffice to show $x\to 0$ as $\theta\to 0$. By Lemma \ref{diophapproxlemma1},
\begin{equation*}
x\ll\frac{1}{|a|H}.
\end{equation*}
Since $|a|\geq 1$ and we chose $H=\Big\lfloor\frac{\sqrt{2}}{\theta^{1/3}}\Big\rfloor$, it follows that $x\ll \theta^{1/3}\to 0$ as $\theta\to 0$.
\end{proof}

\section{
Proof of Theorem \ref{resultirrat}}
\label{irrational}
\subsection{Preparatory results}
The following is a three-dimensional analogue of \cite[Lemma 5.1]{maff2d}. Recall the Definitions \ref{defcap} of a spherical cap and \ref{defseg} of a spherical segment.
\begin{lemma}
\label{cone}
Let $c=c(R)>0$, with $c\to 0$ as $R\to\infty$. Fix a point $B\in R\mathcal{S}^2$, and let $\beta$ be a unit vector. Then all points $B'\in R\mathcal{S}^2$ satisfying $|\langle B-B',\beta\rangle|\leq c|B-B'|$ lie: either on the same spherical segment $S$, of opening angle $\theta=8c+O(c^3)$ and direction $\beta$; or on the same spherical cap, of radius $\ll cR$ and direction $\beta$, on $R\mathcal{S}^2$.
\end{lemma}
\begin{proof}
The condition $|\langle B-B',\beta\rangle|\leq c|B-B'|$ means $B-B'$ and $\beta$ are close to being orthogonal, in the sense that $|\cos(\varphi_{B-B',\beta})|\leq c$, where $0\leq\varphi_{v,w}\leq\pi$ denotes the angle between two non-zero vectors $v,w\in\mathbb{R}^3$. Let $\{s_i\}_i$ be the set of straight lines through $B$ satisfying
\begin{equation*}
|\cos(\varphi_{s_i,\beta})|=c.
\end{equation*}
The lines $\{s_i\}_i$ are the generators of a cone with vertex $B$. Let $\mathcal{R}$ be the region of $\mathbb{R}^3$ delimited by this cone. We then have
\begin{equation*}
\{B'\in\mathbb{R}^3 : |\cos(\varphi_{B-B',\beta})|\leq c\}=\mathbb{R}^3\setminus\mathcal{R}.
\end{equation*}
It follows that
\begin{equation*}
\{B'\in R\mathcal{S}^2 : |\cos(\varphi_{B-B',\beta})|\leq c\}=(\mathbb{R}^3\setminus\mathcal{R})\cap R\mathcal{S}^2=:\mathcal{R}'.
\end{equation*}
We now show that $\mathcal{R}'$ is contained in either a spherical segment or cap. Let $\Pi$ be the plane containing $B$ and $\beta$ (and thus also the origin $O$). The two lines belonging to the set $\{s_i\}_i$ and lying on $\Pi$ will be denoted $s',s''$. Moreover, call $D$ the further intersection between $R\mathcal{S}^2$ and $s'$, meaning $R\mathcal{S}^2\cap s'=\{B,D\}$. Likewise, call $E$ the further intersection between $R\mathcal{S}^2$ and $s''$, meaning $R\mathcal{S}^2\cap s''=\{B,E\}$. Note that possibly one of the lines $s',s''$, say $s''$, is tangent to the sphere $R\mathcal{S}^2$, in which case $E=B$. Let $\Pi_1,\Pi_2$ be planes orthogonal to $\beta$ and through $D,E$ respectively, and denote $\mathcal{B}_1=R\mathcal{S}^2\cap\Pi_1$, $\mathcal{B}_2=R\mathcal{S}^2\cap\Pi_2$. 
By the expansion
\begin{equation*}
\arccos(c)=\frac{\pi}{2}-c+O(c^3) 
\end{equation*}
we have
\begin{equation*}
\varphi_{s',\beta}
=
\varphi_{s'',\beta}
=
\frac{\pi}{2}-c+O(c^3),
\qquad
\varphi_{s',s''}
=
\pi-\varphi_{s',\beta}-\varphi_{s'',\beta}=2c+O(c^3).
\end{equation*}
Let $D',D''$ be points on $s'$ on opposite sides of $B$, and $E',E''$ be points on $s''$ on opposite sides of $B$, so that: $\overline{BD'}=\overline{BD''}=\overline{BE'}=\overline{BE''}=3R$, $D$ lies on $s'$ between $B$ and $D'$, and $\widehat{D'BE'}=\varphi_{s',s''}=2c+O(c^3)$. There are two cases:
\begin{itemize}
\item
In case $E$ lies on $s''$ between $B$ and $E'$, we have $\mathcal{R}'\subset S$, where $S$ is the spherical segment of bases $\mathcal{B}_1, \mathcal{B}_2$. The opening angle of $S$ is
\begin{equation*}
\theta=2\cdot\widehat{DOE}=4\cdot\widehat{D'BE'}=8c+O(c^3).
\end{equation*}
\item
In case $E$ lies on $s''$ between $B$ and $E''$, or in case $E=B$, we have $\mathcal{R}'\subset T$, where $T$ is the spherical cap of direction $\beta$ and base either $\mathcal{B}_1$ or $\mathcal{B}_2$, whichever is the largest. Assume w.l.o.g. that the cap of base $\mathcal{B}_1\ni D$ is the largest. Denoting $H=R\beta\in R\mathcal{S}^2$, the radius of $T$ is
\begin{equation*}
\overline{HD}\leq\overline{BD}=\widehat{DOB}\cdot R\leq 2\cdot\widehat{D'BE'}\cdot R\ll cR.
\end{equation*}
\end{itemize}
\end{proof}

\begin{defin}
Given an integer $m$ which is the sum of three squares, define
\begin{equation*}
\widehat{\mathcal{E}}(m):=\frac{1}{\sqrt{m}}\mathcal{E}(m)\subset S^2
\end{equation*}
to be the projection of the set of lattice points on the unit sphere (cf. \cite[(1.5)]{bosaru} and \cite[(4.3)]{ruwiye}).
\end{defin}
\begin{defin}
For $\sigma>0$, define the \textbf{Riesz $\sigma$-energy} of $N$ (distinct) points $P_1,\dots,P_N$ on $S^2$ as
\begin{equation*}
E_\sigma(P_1,\dots,P_N):=\sum_{i\neq j}\frac{1}{|P_i-P_j|^\sigma}.
\end{equation*}
\end{defin}
\noindent
Bourgain, Sarnak and Rudnick computed the following precise asymptotics for the Riesz $\sigma$-energy of $\widehat{\mathcal{E}}(m)$.
\begin{prop}[Bourgain, Sarnak and Rudnick; see {\cite[Theorem 1.1]{bsr016}}, {\cite[Theorem 1.1]{bosaru}} and {\cite[Theorem 4.1]{ruwiye}}]
\label{asyriesz}
Fix $0<\sigma<2$. Suppose $m\to\infty$, $m\not\equiv 0,4,7 \pmod 8$. Then there is some $\delta>0$ so that
\begin{equation*}
E_\sigma(\widehat{\mathcal{E}}(m))=I(\sigma)\cdot N^2+O(N^{2-\delta})
\end{equation*}
where
\begin{equation*}
I(\sigma)=\frac{2^{1-\sigma}}{2-\sigma}.
\end{equation*}
\end{prop}


\subsection{Proof of Theorem \ref{resultirrat}}
\label{sectpfirrat}
\begin{proof}[Proof of Theorem \ref{resultirrat}]
Apply Proposition \ref{mainprop}, yielding \eqref{firststep}. Let $\rho=\rho(R)$ be a parameter such that $\rho\to 0$ as $R\to\infty$. Let us split the summation on the RHS of \eqref{firststep}, applying \eqref{min}:
\begin{equation}
\label{splitsum}
\text{Var}\left(\frac{\mathcal{Z}}{\sqrt{m}}\right)
\ll
\frac{1}{N^2}
\left[
\sum_{|\langle\mu-\mu',\alpha\rangle|\leq \rho\cdot|\mu-\mu'|}
1
+
\sum_{|\langle\mu-\mu',\alpha\rangle|\geq \rho\cdot|\mu-\mu'|}
\frac{1}{\langle\mu-\mu',\alpha\rangle^2}
\right].
\end{equation}
To bound the first summation on the RHS of \eqref{splitsum}, 
we start by applying Lemma \ref{cone} with $c=\rho$, $B=\mu$ and $\beta=\alpha$: for fixed $\mu$, the condition 
\begin{equation*}
|\langle\mu-\mu',\alpha\rangle|\leq \rho\cdot|\mu-\mu'|
\end{equation*}
means the lattice point $\mu'$ must lie on a spherical segment $S_{\mu}$ of opening angle $
8\rho+O(\rho^3)$ and direction $\alpha$, or on a spherical cap $T_\mu$ of radius $\ll \rho R$ and direction $\alpha$, on $R\mathcal{S}^2$. It follows that
\begin{equation}
\begin{aligned}
\label{1sumbd}
\#\{(\mu,\mu'): |\langle\mu-\mu',\alpha\rangle|\leq \rho\cdot|\mu-\mu'|\}
\\=
\sum_{\mu}\#\{\mu' : \mu'\in T_\mu\}
+
\sum_{\mu}\#\{\mu' : \mu'\in S_\mu\}
\\
\leq 2\cdot\#\{(\mu,\mu'): \mu,\mu'\in T\}+
\sum_{\mu}
\#\{\mu': \mu'\in S_\mu\},
\end{aligned}
\end{equation}
where $T$ is the spherical cap of radius $j\rho R$ (for some large enough $j\in\mathbb{R}^+$) and direction $\alpha$. 
Recalling the notation \eqref{chi(R,s)}, 
we may write
\begin{equation*}
\#\{(\mu,\mu'): \mu,\mu'\in T\}=(\chi(R,j\rho R))^2.
\end{equation*}
If we assume $\rho=o(\frac{1}{R^{3/4}})$ (eventually we are going to choose $\rho=\frac{1}{R^{6/7}}$), we get 
\begin{equation}
\label{1sumbd1}
\#\{(\mu,\mu'): \mu,\mu'\in T\}\ll R^\epsilon
\end{equation}
by Lemma \ref{lemma2.1}.
\\
For each $\mu$, the number of lattice points inside $S_\mu$ is bounded by the maximal number of lattice points $\psi$ (recall \eqref{psi}) in a spherical segment of opening angle $8\rho+O(\rho^3)$. We apply Proposition \ref{lpseg} with $\theta=8\rho+O(\rho^3)$:
\begin{equation}
\label{1sumbd2}
\sum_{\mu}
\#\{\mu': \mu'\in S_\mu\}
\leq
\sum_{\mu}\psi
\ll
N\cdot\kappa(R)(1+R\cdot \rho^{1/3}).
\end{equation}
By substituting \eqref{1sumbd1} and \eqref{1sumbd2} into \eqref{1sumbd}, we get
the following bound for the first summation on the RHS of \eqref{splitsum}:
\begin{equation}
\begin{aligned}
\label{contrib1}
\#\{(\mu,\mu'): |\langle\mu-\mu',\alpha\rangle|\leq \rho\cdot|\mu-\mu'|\}
\ll
R^\epsilon
+
N\cdot\kappa(R)(1+R\cdot \rho^{1/3})
\\
\ll
R^\epsilon N(1+R\cdot \rho^{1/3}),
\end{aligned}
\end{equation}
where we also used \eqref{kappa3bound}.
\\
We now turn to the second summation on the RHS of \eqref{splitsum}. Let $\epsilon'>0$ and apply Proposition \ref{asyriesz} with $\sigma=2-\epsilon'$:
\begin{equation}
\begin{aligned}
\label{contrib2}
\sum_{|\langle\mu-\mu',\alpha\rangle|\geq \rho\cdot|\mu-\mu'|}
\frac{1}{\langle\mu-\mu',\alpha\rangle^2}
\leq
\sum_{|\langle\mu-\mu',\alpha\rangle|\geq \rho\cdot|\mu-\mu'|}
\frac{1}{\rho^2|\mu-\mu'|^2}
\\
\leq
\frac{1}{\rho^2}
\sum_{\mu\neq\mu'}
\frac{1}{|\mu-\mu'|^{2-\epsilon'}}
\sim
\frac{N^2}{\rho^2 R^{2-\epsilon'}}.
\end{aligned}
\end{equation}
Inserting \eqref{contrib1} and \eqref{contrib2} into \eqref{splitsum}, and recalling \eqref{totnumlp}, we deduce
\begin{equation}
\begin{aligned}
\label{deduce}
\text{Var}\left(\frac{\mathcal{Z}}{\sqrt{m}}\right)
\ll
\frac{1}{N^2}
R^\epsilon N\left(\left(1+R\cdot \rho^{1/3}\right)+\frac{N}{\rho^2 R^2}\right)
\\
\ll
\frac{1}{N^2}R^\epsilon N\left(R\cdot \rho^{1/3}+\frac{1}{\rho^2 R}\right).
\end{aligned}
\end{equation}
The optimal choice for the parameter is $\rho=\frac{1}{R^{6/7}}$, and it follows that
\begin{equation*}
\text{Var}\left(\frac{\mathcal{Z}}{\sqrt{m}}\right)
\ll
\frac{N\cdot R^{\frac{5}{7}+\epsilon}}{N^2}
\ll
\frac{1}{m^{\frac{1}{7}-\epsilon}}
.
\end{equation*}
\end{proof}
As mentioned in the Introduction, Theorem \ref{resultirrat} prescribes an unconditional bound for all energies $m$, whereas for the two-dimensional problem, an unconditional bound is only given for a {\em density one sequence} of energies (\cite[Theorem 1.5]{maff2d}), and a bound for all $m$ is given conditionally (\cite[Theorem 1.4]{maff2d}). The reason for this is the significant difference between the total number of lattice points on a sphere and on a circle (recall \eqref{totnumlp} and \eqref{lpcircles}).

\section{
Proof of Theorem \ref{resulthalfrat}}
\label{halfrational}
For lines satisfying $\frac{\alpha_2}{\alpha_1}\in\mathbb{Q}$ and $\frac{\alpha_3}{\alpha_1}\in\mathbb{R}\setminus\mathbb{Q}$, we may unconditionally improve our bound for the variance of nodal intersections (Theorem \ref{resultirrat}) by gaining on the bound for the number of lattice points in a spherical segment of direction $\alpha$ (compare Propositions \ref{lpseg} and \ref{lpsegbis}); this is because we approximate one irrational number instead of two simultaneously (compare Lemmas \ref{diophapproxlemma1} and \ref{diophapproxlemma2}).

\subsection{Diophantine approximation}
\begin{prop}[Dirichlet]
\label{dirich}
Given $\zeta\in\mathbb{R}\setminus\mathbb{Q}$ and an integer $H\geq 1$, there exist $p,q\in\mathbb{Z}$ so that $1\leq q\leq H$ and
\begin{equation*}
\left|\zeta-\frac{p}{q}\right|
<
\frac{1}{q H}.
\end{equation*}
\end{prop}

\begin{lemma}
\label{diophapproxlemma2}
Let $\alpha\in\mathbb{R}^3$ with $|\alpha|=1$ and satisfying $\frac{\alpha_2}{\alpha_1}\in\mathbb{Q}$ and $\frac{\alpha_3}{\alpha_1}\in\mathbb{R}\setminus\mathbb{Q}$. Write $\frac{\alpha_2}{\alpha_1}=\frac{u}{v}$ with $u,v\in\mathbb{Z}$ and $v>0$. Define
\begin{equation}
\label{deftau}
\tau=\tau_\alpha:=\max\left(|u|,v,\frac{1}{|\alpha_1|}\right)+1.
\end{equation}
Then for all integers $H\geq 1$, there exists $a\in\mathbb{Z}^3$ satisfying
\begin{numcases}{}
\label{syst3cond1}
|a|<\sqrt{3}\tau_\alpha^2H
\\
\label{syst3cond2}
\left|\alpha-\frac{a}{|a|}\right|
<2\sqrt{3}\cdot\frac{\tau_\alpha^2}{|a|H}.
\end{numcases}
\end{lemma}
\begin{proof}
Take $\zeta=\frac{\alpha_3}{\alpha_1}$ and a large integer $H$. By Proposition \ref{dirich}, there exist $p,q\in\mathbb{Z}$ so that $1\leq q\leq H$ and
\begin{equation*}
\left|\frac{\alpha_3}{\alpha_1}-\frac{p}{q}\right|
<
\frac{1}{q H}.
\end{equation*}
Assume $\alpha_1>0$ (in case $\alpha_1<0$, take $-\alpha$). Fix $a:=(qv,qu,pv)$ and let us show this vector satisfies both \eqref{syst3cond1} and \eqref{syst3cond2}. We have
\begin{gather*}
\left|\frac{p}{q}\right|
\leq\left|\frac{\alpha_3}{\alpha_1}\right|
+\frac{1}{qH}
<\frac{1}{\alpha_1}+1
\leq\tau
\Rightarrow
|p|<\tau q
\\
\Rightarrow
|a|^2=q^2v^2+q^2u^2+p^2v^2<\tau^2q^2+\tau^2q^2+\tau^4q^2<3\tau^4q^2
\end{gather*}
so that
\begin{equation}
\label{tau}
|a|<\sqrt{3}\tau^2q\leq\sqrt{3}\tau^2H
\end{equation}
and \eqref{syst3cond1} is verified. We now turn to proving \eqref{syst3cond2}. We define the vector $d:=\frac{\alpha_1}{qv}\cdot a$, hence (as $\alpha_1>0$)
\begin{equation*}
\frac{d}{|d|}=
\frac
{\frac{\alpha_1}{qv}\cdot a}
{\frac{|\alpha_1|}{|qv|}\cdot |a|}
=
\frac{a}{|a|}.
\end{equation*}
Apply Lemma \ref{trineq} with $w=\alpha$ and $v=d$, recalling $|\alpha|=1$:
\begin{equation}
\label{tau1}
\left|\alpha-\frac{a}{|a|}\right|
=
\left|\frac{\alpha}{|\alpha|}-\frac{d}{|d|}\right|
\leq
2\frac{|\alpha-d|}{|\alpha|}
=
2|\alpha-d|.
\end{equation}
Moreover,
\begin{equation}
\begin{aligned}
\label{tau2}
|\alpha-d|
=
\left|
\alpha-\frac{\alpha_1}{qv}\cdot a
\right|
=
\left|
\left(
\alpha_1-\frac{\alpha_1}{qv}\cdot qv,
\alpha_2-\frac{\alpha_1}{qv}\cdot qu,
\alpha_3-\frac{\alpha_1}{qv}\cdot pv
\right)
\right|
\\
=
|\alpha_1|
\cdot
\left|
\left(
0,
\frac{\alpha_2}{\alpha_1}-\frac{u}{v},
\frac{\alpha_3}{\alpha_1}-\frac{p}{q}
\right)
\right|
=
|\alpha_1|
\cdot
\left|
\left(
0,
0,
\frac{\alpha_3}{\alpha_1}-\frac{p}{q}
\right)
\right|
<
\frac{1}{q H}.
\end{aligned}
\end{equation}
By \eqref{tau},
\begin{equation}
\label{tau3}
\frac{1}{q}<\frac{\sqrt{3}\tau^2}{|a|}.
\end{equation}
Combining \eqref{tau1}, \eqref{tau2} and \eqref{tau3},
\begin{equation*}
\left|\alpha-\frac{a}{|a|}\right|
\leq
2|\alpha-d|
<
2\cdot\frac{1}{q H}
<
2\sqrt{3}\cdot\frac{\tau^2}{|a|H}
\end{equation*}
and \eqref{syst3cond2} is verified.
\end{proof}
\noindent
Recall the notation \eqref{psi} for $\psi$, the maximal number of lattice points in a spherical segment.
\begin{prop}
\label{lpsegbis}
Let $S\subset R\mathcal{S}^2$ be a spherical segment of opening angle $\theta$, radius of larger base $k$, and direction $\beta$, with $\frac{\beta_2}{\beta_1}\in\mathbb{Q}$ and $\frac{\beta_3}{\beta_1}\in\mathbb{R}\setminus\mathbb{Q}$. Then the maximal number of lattice points lying on $S$ satisfies
\begin{equation*}
\psi\ll_{\beta} \kappa(R)(1+R\cdot\theta^{1/2})
\end{equation*}
for $\theta\to 0$.
\end{prop}
\begin{proof}
Recall \eqref{mezzo}:
\begin{equation*}
\psi\ll \kappa(R)\cdot\left[1+R\left(|a|\theta+|a|\left|\beta-\frac{a}{|a|}\right|\right)\right].
\end{equation*}
It then suffices to bound
\begin{equation*}
|a|\theta+|a|\left|\beta-\frac{a}{|a|}\right|.
\end{equation*}
We apply Lemma \ref{diophapproxlemma2} with $\alpha=\beta$: for all integers $H\geq 1$, there exists $a=(a_1,a_2,a_3)$ so that
\begin{equation*}
|a|\theta+|a|\left|\beta-\frac{a}{|a|}\right|
<\sqrt{3}\tau_\beta^2H\cdot\theta+2\sqrt{3}\cdot\frac{\tau_\beta^2}{H}
\ll_{\beta}
H\cdot\theta+\frac{1}{H},
\end{equation*}
with $\tau_\beta$ as in \eqref{deftau}. The tradeoff gives us the choice $H=\left\lfloor\frac{1}{\theta^{1/2}}\right\rfloor$, hence
\begin{equation}
\label{mezzo2}
|a|\theta+\left|\beta-\frac{a}{|a|}\right|\ll_{\beta} \theta^{1/2}.
\end{equation}
The statement of the present proposition follows on inserting \eqref{mezzo2} into \eqref{mezzo}.
\end{proof}

\subsection{Proof of Theorem \ref{resulthalfrat}}
\label{sectpfhalfrat}
\begin{proof}[Proof of Theorem \ref{resulthalfrat}]
We will follow the proof of Theorem \ref{resultirrat}, except the maximal number of lattice points in spherical segments of opening angle $\theta$ will be bounded via Proposition \ref{lpsegbis}. Let $\rho=\rho(R)$ be a parameter such that $\rho\to 0$ as $R\to\infty$. We need to bound the two summations on the RHS of \eqref{splitsum}. For the former, we use \eqref{1sumbd} and \eqref{1sumbd1}; we gain on the estimate \eqref{1sumbd2} by invoking Proposition \ref{lpsegbis} with $\theta=8\rho+O(\rho^3)$:
\begin{equation}
\label{1sumbd2bis}
\sum_{\mu}
\#\{\mu': \mu'\in S_\mu\}
\leq
\sum_{\mu}\psi
\ll
N\cdot\kappa(R)(1+R\cdot \rho^{1/2}).
\end{equation}
By substituting \eqref{1sumbd1} and \eqref{1sumbd2bis} into \eqref{1sumbd}, we get the following bound for the first summation on the RHS of \eqref{splitsum}:
\begin{equation}
\begin{aligned}
\label{contrib1bis}
\#\{(\mu,\mu'): |\langle\mu-\mu',\alpha\rangle|\leq \rho\cdot|\mu-\mu'|\}
\ll
R^\epsilon
+
N\cdot\kappa(R)(1+R\cdot \rho^{1/2})
\\
\ll
R^\epsilon N(1+R\cdot \rho^{1/2}).
\end{aligned}
\end{equation}
For the second summation on the RHS of \eqref{splitsum}, we have the bound \eqref{contrib2}. Inserting \eqref{contrib1bis} and \eqref{contrib2} into \eqref{splitsum}, and recalling \eqref{totnumlp}, we deduce
\begin{equation*}
\begin{aligned}
\text{Var}\left(\frac{\mathcal{Z}}{\sqrt{m}}\right)
\ll
\frac{1}{N^2}
R^\epsilon N\left(\left(1+R\cdot \rho^{1/2}\right)+\frac{N}{\rho^2 R^2}\right)
\\
\ll
\frac{1}{N^2}R^\epsilon N\left(R\cdot \rho^{1/2}+\frac{1}{\rho^2 R}\right).
\end{aligned}
\end{equation*}
The optimal choice for the parameter is $\rho=\frac{1}{R^{4/5}}$, thus
\begin{equation*}
\text{Var}\left(\frac{\mathcal{Z}}{\sqrt{m}}\right)
\ll
\frac{N\cdot R^{\frac{3}{5}+\epsilon}}{N^2}
\ll
\frac{1}{m^{\frac{1}{5}-\epsilon}}
.
\end{equation*}
\end{proof}

\section{Conditional result: proof of Theorem \ref{resultcond}}
\label{conditional}
Recall the Definitions \ref{defcap} of a spherical cap and \ref{defseg} of a spherical segment.
\begin{lemma}
\label{twoparpla}
Given $0<c<R$, fix a point $B\in R\mathcal{S}^2$, and let $\beta$ be a unit vector. Then all points $B'\in R\mathcal{S}^2$ satisfying $|\langle B-B',\beta\rangle|\leq c$ lie either on the same spherical segment, of height $2c$ and direction $\beta$, or on the same spherical cap, of height at most $2c$ and direction $\beta$, on $R\mathcal{S}^2$.
\end{lemma}
\begin{proof}
For a real number $\xi$, define the plane
\begin{equation*}
\Pi_{\xi}:\langle\beta,(x,y,z)\rangle=\xi,
\end{equation*}
orthogonal to $\beta$. For $-c\leq c'\leq c$, the condition
\begin{equation*}
\langle B-B',\beta\rangle=c'
\Leftrightarrow
\langle\beta,B'\rangle=\langle\beta,B\rangle-c'
\end{equation*}
means $B'$ lies on the plane $\Pi_{\langle\beta,B\rangle-c'}$. Therefore, all $B'$ satisfying $|\langle B-B',\beta\rangle|\leq c$ belong to a region $\mathcal{R}$ of $\mathbb{R}^3$ delimited by two parallel planes, namely $\Pi_{\langle\beta,B\rangle-c}$ and $\Pi_{\langle\beta,B\rangle+c}$. The distance between these two planes is $2c$. Denote
\begin{equation*}
\mathcal{R'}=\mathcal{R}\cap R\mathcal{S}^2.
\end{equation*}
In case $|\langle\beta,B\rangle|<R-c$, both $\Pi_{\langle\beta,B\rangle-c}$ and $\Pi_{\langle\beta,B\rangle+c}$ intersect $R\mathcal{S}^2$ in a circle. By Definition \ref{defseg}, $\mathcal{R'}$ is then a spherical segment, of height $2c$ and direction $\beta$.
In case $R-c\leq|\langle\beta,B \rangle|\leq R$, one of the intersections $\Pi_{\langle\beta,B\rangle-c}\cap R\mathcal{S}^2$ and $\Pi_{\langle\beta,B\rangle+c}\cap R\mathcal{S}^2$ is either empty or a single point. By Definition \ref{defcap}, $\mathcal{R'}$ is then a spherical cap, of height at most $2c$ and direction $\beta$.
\end{proof}
\begin{proof}[Proof of Theorem \ref{resultcond}]
Apply Proposition \ref{mainprop}, yielding \eqref{firststep}. Let $0<\rho<R$ be  a parameter and split the summation on the RHS of \eqref{firststep}, applying \eqref{min}:
\begin{equation}
\label{splitsum2}
\text{Var}\left(\frac{\mathcal{Z}}{\sqrt{m}}\right)\ll
\frac{1}{N^2}\cdot
\left(
\sum_{|\langle\mu-\mu',\alpha\rangle|\leq \rho}
1
+
\sum_{|\langle\mu-\mu',\alpha\rangle|\geq \rho}
\frac{1}{\langle\mu-\mu',\alpha\rangle^2}
\right).
\end{equation}
For the second summation on the RHS of \eqref{splitsum2}, we write
\begin{equation}
\label{1stsum}
\sum_{|\langle\mu-\mu',\alpha\rangle|\geq \rho}
\frac{1}{\langle\mu-\mu',\alpha\rangle^2}
\leq
\frac{1}{\rho^2}\sum_{|\langle\mu-\mu',\alpha\rangle|\geq \rho}1
\leq
\frac{N^2}{\rho^2}.
\end{equation}
For the remaining summation in \eqref{splitsum2}, we show that there are few pairs $(\mu,\mu')$ satisfying $|\langle\mu-\mu',\alpha\rangle|\leq \rho$. Fix a lattice point $\mu$ and apply Lemma \ref{twoparpla} with $\beta=\alpha$ and $c=\rho$: then $\mu'$ verifies $|\langle\mu-\mu',\alpha\rangle|\leq \rho$ if and only if it lies on a spherical segment $S_\mu$ of height $2\rho$ and direction $\alpha$, or on a spherical cap $T_\mu$ of height at most $2\rho$ and direction $\alpha$. That is to say,
\begin{equation}
\begin{aligned}
\label{2ndsum}
\#\{(\mu,\mu'): |\langle\mu-\mu',\alpha\rangle|\leq \rho\}
\leq
\sum_{\mu}\#\{\mu' : \mu'\in T_\mu\}
+
\sum_{\mu}\#\{\mu' : \mu'\in S_\mu\}
\\
\leq
2\cdot\#\{(\mu,\mu'): \mu,\mu'\in T\}+
\sum_{\mu}
\#\{\mu': \mu'\in S_\mu\},
\end{aligned}
\end{equation}
where $T$ is the spherical cap of height $2\rho$ and direction $\alpha$.
By Conjecture \ref{brgafaconj}, the maximal number of lattice points in a cap of radius $s$ of the sphere $R\mathcal{S}^2$ satisfies
$\chi(R,s)
\ll
R^\epsilon(1+\frac{s^2}{R})$.
Therefore, recalling \eqref{shR},
\begin{equation*}
\#\{\mu: \mu\in T\}
\ll
R^\epsilon\left(1+\frac{s^2}{R}\right)
\ll
R^\epsilon(1+\rho),
\end{equation*}
and it follows that
\begin{equation}
\label{2ndsumcaps}
\#\{(\mu,\mu'): \mu,\mu'\in T\}
\ll
R^\epsilon(1+\rho^2).
\end{equation}
To bound the number of lattice points in the spherical segment $S_\mu$, we may apply Corollary \ref{covercapscor2} with $h=2\rho$. We then get
\begin{equation}
\label{2ndsumsegments}
\sum_{\mu}
\#\{\mu': \mu'\in S_\mu\}
\ll
N\cdot R^\epsilon\cdot (R^{1/2}+\rho).
\end{equation}
Inserting the estimates \eqref{2ndsumcaps} and \eqref{2ndsumsegments} into \eqref{2ndsum}, and then the inequalities \eqref{2ndsum} and \eqref{1stsum} into \eqref{splitsum2} gives us
\begin{gather*}
\text{Var}\left(\frac{\mathcal{Z}}{\sqrt{m}}\right)
\ll
\frac{1}{N^2}\cdot\left(
\frac{N^2}{\rho^2}
+
R^\epsilon
\left(1+\rho^2+NR^{1/2}+N\rho\right)\right).
\end{gather*}
Taking any $\rho$ in the range $R^{\frac{1}{4}}\leq \rho\leq R^{\frac{1}{2}}$ and recalling \eqref{totnumlp}, we obtain
\begin{equation*}
\text{Var}\left(\frac{\mathcal{Z}}{\sqrt{m}}\right)
\ll
\frac{N\cdot R^{\frac{1}{2}+\epsilon}}{N^2}
\ll
\frac{1}{m^{\frac{1}{4}-\epsilon}}
.
\end{equation*}
\end{proof}

\section{Proof of auxiliary results}
\label{auxpfs}
We begin by proving Lemma \ref{rsq}.
Recall the notation for the lattice point set $\mathcal{E}$ and number $N$, and the definition \eqref{2ndmom} of $\mathcal{R}_2(m)$.
\begin{proof}[Proof of Lemma \ref{rsq}]
Recall that $r=r(t_1,t_2)$ is the covariance function \eqref{r} restricted to $\mathcal{C}$, and the notation
\begin{equation*}
r_1=\frac{\partial r(t_1,t_2)}{\partial t_1},
\qquad
r_2=\frac{\partial r(t_1,t_2)}{\partial t_2}
\qquad
\text{and}
\quad
r_{12}=\frac{\partial^2 r(t_1,t_2)}{\partial t_1\partial t_2}.
\end{equation*}
We will show
\begin{equation}
\label{rbound}
\int_0^L \int_0^L 
r^2(t_1,t_2)
dt_1dt_2
=\frac{1}{N^2}
\sum_{(\mu,\mu')\in\mathcal{E}^2}
\left|\int_0^L e^{2\pi it\langle\mu-\mu',\alpha\rangle}dt\right|^2,
\end{equation}
\begin{equation}
\label{r1bound}
\int_0^L \int_0^L 
\left(\frac{r_i(t_1,t_2)}{\sqrt{m}}\right)^2
dt_1dt_2
\leq \frac{1}{N^2}
\sum_{(\mu,\mu')\in\mathcal{E}^2}
\left|\int_0^L e^{2\pi it\langle\mu-\mu',\alpha\rangle}dt\right|^2,
\end{equation}
for $i=1,2$, and
\begin{equation}
\label{r12bound}
\int_0^L \int_0^L 
\left(\frac{r_{12}(t_1,t_2)}{m}\right)^2
dt_1dt_2
\leq \frac{1}{N^2}
\sum_{(\mu,\mu')\in\mathcal{E}^2}
\left|\int_0^L e^{2\pi it\langle\mu-\mu',\alpha\rangle}dt\right|^2.
\end{equation}
By squaring the covariance function we find
\begin{gather*}
\int_0^L \int_0^L |r(t_1,t_2)|^2dt_1dt_2
=
\\
\int_0^L \int_0^L
\frac{1}{N^2}
\sum_{(\mu,\mu')\in\mathcal{E}^2}
e^{2\pi i(t_1-t_2)\langle\mu-\mu',\alpha\rangle}
dt_1dt_2
=
\\
\frac{1}{N^2}
\sum_{(\mu,\mu')\in\mathcal{E}^2}
\left|\int_0^L e^{2\pi it\langle\mu-\mu',\alpha\rangle}dt\right|^2,
\end{gather*}
proving \eqref{rbound}. Next,
\begin{equation*}
r_1=\frac{\partial r(t_1,t_2)}{\partial t_1}
=
\frac{1}{N}\sum_{\mu\in\mathcal{E}}2\pi i\langle\mu,\alpha\rangle e^{2\pi i(t_1-t_2)\langle\mu,\alpha\rangle}
\end{equation*}
and by Cauchy-Schwartz,
\begin{gather*}
\int_0^L \int_0^L \left|\frac{r_1}{2\pi\sqrt{ m}}\right|^2dt_1dt_2
\\
=
\int_0^L \int_0^L \frac{1}{N^2}
\sum_{(\mu,\mu')\in\mathcal{E}^2}
\left\langle\frac{\mu}{|\mu|},\alpha\right\rangle \left\langle\frac{\mu'}{|\mu'|},\alpha\right\rangle 
e^{2\pi i(t_1-t_2)\langle\mu,\alpha\rangle}
e^{2\pi i(t_1-t_2)\langle\mu',\alpha\rangle}
dt_1dt_2
\\
\leq
\int_0^L \int_0^L \frac{1}{N^2}
\sum_{(\mu,\mu')\in\mathcal{E}^2} 
e^{2\pi i(t_1-t_2)\langle\mu,\alpha\rangle}
e^{2\pi i(t_1-t_2)\langle\mu',\alpha\rangle}
dt_1dt_2
\\
=
\frac{1}{N^2}
\sum_{(\mu,\mu')\in\mathcal{E}^2} \left|\int_0^L e^{2\pi it\langle\mu-\mu',\alpha\rangle}dt\right|^2,
\end{gather*}
yielding \eqref{r1bound}. For the second mixed derivative:
\begin{equation*}
r_{12}=\frac{\partial^2 r(t_1,t_2)}{\partial t_1\partial t_2}
=
\frac{1}{N}\sum_{\mu\in\mathcal{E}}(2\pi i)^2\langle\mu,\alpha\rangle^2 e^{2\pi i(t_1-t_2)\langle\mu,\alpha\rangle}
\end{equation*}
thus, again by Cauchy-Schwartz,
\begin{gather*}
\int_0^L \int_0^L \left|\frac{r_{12}}{4\pi^2 m}\right|^2dt_1dt_2
\\
=
\int_0^L \int_0^L \frac{1}{N^2}
\sum_{(\mu,\mu')\in\mathcal{E}^2}
\left\langle\frac{\mu}{|\mu|},\alpha\right\rangle^2 \left\langle\frac{\mu'}{|\mu'|},\alpha\right\rangle^2 
e^{2\pi i(t_1-t_2)\langle\mu,\alpha\rangle}
e^{2\pi i(t_1-t_2)\langle\mu',\alpha\rangle}
dt_1dt_2
\\
\leq
\frac{1}{N^2}
\sum_{(\mu,\mu')\in\mathcal{E}^2}
\left|\int_0^L e^{2\pi it\langle\mu-\mu',\alpha\rangle}dt\right|^2,
\end{gather*}
and \eqref{r12bound} follows.
\end{proof}
\noindent
Next, we prove Lemma \ref{kthetah}.
\begin{proof}[Proof of Lemma \ref{kthetah}]
We write
\begin{equation*}
S=T_2\setminus T_1
\end{equation*}
where $T_1$ and $T_2$ are spherical caps of
heights $h_1,h_2$, radii of bases $k_1,k_2$, and opening angles $\theta_1,\theta_2$ respectively; note that $h=h_2-h_1$ and $k_2=k$. Inserting $h_2\ll R$ into the relation \eqref{shR} for the cap $T_2$ yields
\begin{equation}
\label{kRh}
k_2\asymp\sqrt{R}\sqrt{h_2}.
\end{equation}
In case $h\asymp h_2$, we immediately have, by \eqref{kRh} and \eqref{htheta},
\begin{equation*}
k\theta\leq k_2\theta_2\ll\sqrt{R}\sqrt{h_2}\arcsin\left(\sqrt{h_2/2R}\right)\ll h_2\asymp h,
\end{equation*}
which proves the lemma in this case.

The remaining case is $h=o(h_2)$: here we may write $h_2=a+b$, $h_1=a-b$, and $h=2b$, with $0\leq b<a\leq R$ and $b(R)=o(a(R))$ as $R\to\infty$. By \eqref{htheta},
\begin{equation*}
\theta=\theta_2-\theta_1=4\left[\arcsin\left(\sqrt{h_2/2R}\right)-\arcsin\left(\sqrt{h_1/2R}\right)\right].
\end{equation*}
By the expansion of $\arcsin$ around $0$,
\begin{multline*}
\theta=4\sum_{n=0}^{\infty}\frac{\binom{2n}{n}}{4^n(2n+1)}\cdot\frac{(\sqrt{h_2})^{2n+1}-(\sqrt{h_1})^{2n+1}}{(\sqrt{2R})^{2n+1}}
\\=2\sqrt{2}\sum_{n=0}^{\infty}\frac{\binom{2n}{n}}{8^n(2n+1)}\cdot\frac{(\sqrt{a+b})^{2n+1}-(\sqrt{a-b})^{2n+1}}{(\sqrt{R})^{2n+1}}.
\end{multline*}
Multiplying and dividing by the quantity
\begin{equation*}
(\sqrt{a+b})^{2n+1}+(\sqrt{a-b})^{2n+1}
\end{equation*}
we obtain
\begin{multline}
\label{thetabound}
\theta\ll \sum_{n=0}^{\infty}\frac{\binom{2n}{n}}{8^n(2n+1)}\cdot\frac{(a+b)^{2n+1}-(a-b)^{2n+1}}{(aR)^{n+\frac{1}{2}}}
\\\ll
\sum_{n=0}^{\infty}\frac{\binom{2n}{n}}{8^n(2n+1)}\cdot\frac{2\binom{2n+1}{1}a^{2n}b}{(aR)^{n+\frac{1}{2}}}
=
\frac{h}{\sqrt{R}\sqrt{a}}\sum_{n=0}^{\infty}\frac{\binom{2n}{n}}{8^n}\cdot\left(\frac{a}{R}\right)^n
.
\end{multline}
By \eqref{kRh} and \eqref{thetabound}, we have
\begin{equation*}
k\theta
\ll h\sum_{n=0}^{\infty}\frac{\binom{2n}{n}}{8^n}\cdot\left(\frac{a}{R}\right)^n
\leq h\sum_{n=0}^{\infty}\frac{\binom{2n}{n}}{8^n}=h\sqrt{2}.
\end{equation*}
\end{proof}
\noindent
We finish the section with the proofs of Lemmas \ref{trineq} and \ref{lemmaclaim}.
\begin{proof}[Proof of Lemma \ref{trineq}]
Let $v,w\in\mathbb{R}^n$ be non-zero. By the triangle inequality:
\begin{equation}
\label{trian1}
\left|\frac{v}{|v|}|w|-w\right|=
\left|\frac{v}{|v|}|w|-v+v-w\right|\leq
\left|\frac{v}{|v|}|w|-v\right|+|v-w|.
\end{equation}
Applying the triangle inequality again,
\begin{equation}
\begin{aligned}
\label{trian2}
\left|\frac{v}{|v|}|w|-v\right|=
\left|\frac{v}{|v|}|w|-\frac{v}{|v|}|v|\right|=
\left|\frac{v}{|v|}\right|\cdot\left||w|-|v|\right|
\\=
\left||w|-|v|\right|\leq|v-w|.
\end{aligned}
\end{equation}
Substituting \eqref{trian2} into \eqref{trian1} we get
\begin{equation*}
\left|\frac{v}{|v|}|w|-w\right|
\leq
2\cdot|v-w|
\Rightarrow
\left|\frac{v}{|v|}-\frac{w}{|w|}\right|
\leq
2\frac{|v-w|}{|w|}.
\end{equation*}
\end{proof}

\begin{proof}[Proof of Lemma \ref{lemmaclaim}]
Fix a vector $a=(a_1,a_2,a_3)\in\mathbb{Z}^3$, and let $\varphi$ be the angle between $\beta$ and $a$ (with $a_1,a_2,a_3$ parameters).
Let $\mathcal{B}_1,\mathcal{B}_2$ be the bases of $S$ (the latter being the larger), lying on the planes $\Pi_1,\Pi_2$ respectively.
Denote $O$ the origin, 
$U=R\beta\in R\mathcal{S}^2$ 
and $V=R\frac{a}{|a|}\in R\mathcal{S}^2$. 
\\
With the same notation as Definition \ref{defseg2}, 
call $\Gamma$ the great circle through $U$ and $V$. All arcs mentioned in this proof lie on the great circle $\Gamma$. Let $\{A,B\}:=\mathcal{B}_1\cap\Gamma$, $\{C,D\}:=\mathcal{B}_2\cap\Gamma$ (so that $\overline{AV}<\overline{BV}$ and $\overline{CV}<\overline{DV}$). 
As the opening angle of the spherical segment $S$ is 
\begin{equation*}
\theta=\widehat{AOC}+\widehat{BOD}=2\cdot\widehat{AOC},
\end{equation*}
and the radius of the circle $\Gamma$ is $R$, we have $\overset{ \frown}{AC}=R\cdot\widehat{AOC}=\frac{R\theta}{2}$.
There are three cases:
\begin{itemize}
\item
\underline{Case 1: 
$\overset{\frown}{UV}<\overset{\frown}{UA}$}.
\\
We shall consider a new spherical segment $S'$, of direction $\frac{a}{|a|}$, and containing $S$; let $S'$ be delimited by the following two planes: $\Pi_1'$ is defined to be orthogonal to $a$, and $A\in\Pi_1'$, while $\Pi_2'$ is defined to be orthogonal to $a$, and $D\in\Pi_2'$.
Denote $\psi$ and $\psi'$ the number of lattice points in $S$ and in $S'$ respectively. Then we have
\begin{equation}
\label{step1}
\psi\leq \psi'.
\end{equation}
Since the direction of $S'$ is the rational vector $\frac{a}{|a|}$, we may use Proposition \ref{ratsect}:
\begin{equation}
\label{step2}
\psi'\leq \kappa(R)\cdot(1+|a|\cdot h'),
\end{equation}
with $\kappa(R)$ as in Definition \ref{kappa} and $h'$ the height of $S'$. To estimate $h'$, we start by considering $S'$ as the disjoint union of two spherical segments $S_1,S_2$ as follows. The plane $\Pi_3'$ is defined to be orthogonal to $a$, with $C\in\Pi_3'$. Let $S_1$ be the segment delimited by $\Pi_1',\Pi_3'$; let $S_2$ be the segment delimited by $\Pi_3',\Pi_2'$. If we denote $h_1$ and $h_2$ the heights of $S_1,S_2$ respectively, then $h'=h_1+h_2$. We have $h_1<\overset{\frown}{AC}=\frac{R\theta}{2}$, and we will now show $h_2<2R\varphi$, hence
\begin{equation*}
h'=h_1+h_2\ll R(\theta+\varphi)
\end{equation*}
which together with \eqref{step1} and \eqref{step2} yield \eqref{claim}. It remains to prove $h_2<2R\varphi$: denote $W$ the point satisfying
\begin{equation*}
\{D,W\}=\Pi_2'\cap\Gamma.
\end{equation*}
Then $\overset{\frown}{CW}$ is an arc on $\Gamma$. We have $\widehat{CDW}=\widehat{UOV}=\varphi$, since $\overline{CD}\perp \overline{OU}$ and $\overline{DW}\perp \overline{OV}$. The height $h_2$ of $S_2$ is less than
\begin{equation*}
\overset{\displaystyle{ \frown}}{CW}=R\cdot\widehat{COW}=2R\cdot\widehat{CDW}=2R\varphi.
\end{equation*}

\item
\underline{Case 2:
$\overset{\frown}{UA}\leq\overset{\frown}{UV}\leq\overset{\frown}{UC}$}.
\\
Denote $\Pi'$ the plane orthogonal to $a$ and containing $D$. The spherical cap $T$ delimited by $\Pi'$ has direction $\frac{a}{|a|}$ and contains $S$. Therefore, the number $\psi$ of lattice points in $S$ cannot exceed the number in $T$, which we will denote $\chi$:
\begin{equation}
\label{step1bis}
\psi\leq \chi.
\end{equation}
Since the direction of $T$ is the rational vector  $\frac{a}{|a|}$, we may use \cite[(2.13)]{brgafa}: as the opening angle of $T$ is $\widehat{VOD}$, we have
\begin{equation}
\label{step2bis}
\chi
\ll \kappa(R)\cdot\left[1+R|a|(\widehat{VOD})^2\right].
\end{equation}
We need to estimate $\widehat{VOD}$.
\begin{gather*}
\widehat{VOD}
=
\widehat{VOU}+\widehat{UOD}
=
\widehat{VOU}+\widehat{UOC}
=
\widehat{VOU}+\widehat{VOU}+\widehat{VOC}
\\
\leq
2\widehat{VOU}+\widehat{AOC}
=
2\varphi+\theta/2.
\end{gather*}
The latter inequality holds because we are assuming $\overset{\frown}{UA}\leq\overset{\frown}{UV}\leq\overset{\frown}{UC}$. By \eqref{step1bis} and \eqref{step2bis}, we find
\begin{gather*}
\psi
\ll \kappa(R)\cdot\left[1+R|a|(\widehat{VOD})^2\right]
\leq \kappa(R)\cdot\left[1+R|a|(2\varphi+\theta/2)^2\right]
\\
\ll
\kappa(R)\cdot\left[1+R|a|(\theta+\varphi)\right],
\end{gather*}
hence \eqref{claim}.

\item
\underline{Case 3:
$\overset{\frown}{UC}<\overset{\frown}{UV}$}.
\\
Consider the cap $T$ of Case 2, of direction $\frac{a}{|a|}$ and containing $S$. We have \eqref{step1bis}, \eqref{step2bis} and, since $\overset{\frown}{UC}<\overset{\frown}{UV}$,
\begin{equation}
\label{step3}
\widehat{VOD}
=
\widehat{VOU}+\widehat{UOD}
=
\widehat{VOU}+\widehat{UOC}
<
\widehat{VOU}+\widehat{VOU}
=
2\varphi.
\end{equation}
By \eqref{step1bis}, \eqref{step2bis} and \eqref{step3},
\begin{gather*}
\psi
\ll \kappa(R)\cdot\left[1+R|a|(\widehat{VOD})^2\right]
< \kappa(R)\cdot\left[1+R|a|(2\varphi)^2\right]
\\
\ll
\kappa(R)\cdot\left[1+R|a|(\theta+\varphi)\right],
\end{gather*}
implying \eqref{claim}.
\end{itemize}
\end{proof}

\section*{Acknowledgements}
This work was carried out as part of the author's PhD thesis at King's College London, under the supervision of Igor Wigman. The author's PhD is funded by a Graduate Teaching Scholarship, Department of Mathematics. The author wishes to thank Igor Wigman for his invaluable guidance, remarks and corrections, and for his availability. The author wishes to thank Ze{\'e}v Rudnick for suggesting this very interesting problem, and for helpful communications and corrections. The author would like to thank an anonymous referee for helpful remarks and suggestions.

\addcontentsline{toc}{section}{References}
\bibliographystyle{plain}
\bibliography{bibfile}

\end{document}